\documentclass[12pt]{amsart}

\usepackage{geometry}
\usepackage{amsmath}
\usepackage{amssymb}
\usepackage{amsthm}
\usepackage{amsfonts, dsfont}
\usepackage{mathrsfs} 
\usepackage{algorithm}
\usepackage{algorithmic}
\usepackage{paralist}
\usepackage{graphics} 
\usepackage{epsfig} 
\usepackage{graphicx}  
\usepackage{epstopdf}
\usepackage{epstopdf}
\usepackage{verbatim}
\usepackage{mathrsfs}
\usepackage{mathtools}
\usepackage{pstricks}
\usepackage{relsize}
\usepackage{tikz}
\usetikzlibrary{matrix}
\usepackage{subcaption}
\usepackage{pgfplots}
\usepackage{fixltx2e}
\usepackage{enumitem}
\usepackage{bm}
\usepackage{upgreek}
\usepackage{url}
\usepackage[colorlinks=true]{hyperref}
\hypersetup{urlcolor=blue, linkcolor=blue, citecolor=red}
\usepackage{cleveref}
\usepackage{bm}
\usepackage{appendix}

\usepackage{tcolorbox} 

\allowdisplaybreaks

\newcommand{\D}[1]{\mbox{\rm #1}} 
\newcommand{\dd}{\D{d}}

\definecolor{ao(english)}{rgb}{0.0, 0.5, 0.0}

\DeclareMathOperator*{\argmin}{argmin}

\DeclareMathOperator*{\essinf}{ess-inf}

\usepackage[abbrev]{amsrefs}

\usepackage{amssymb}

\numberwithin{equation}{section}

\newtheorem{theorem}{Theorem}[section]
\newtheorem{lemma}[theorem]{Lemma}
\newtheorem{assumption}[theorem]{Assumption}
\newtheorem{corollary}[theorem]{Corollary}

\newtheorem{remark}[theorem]{Remark}
\newtheorem{definition}[theorem]{Definition}

\definecolor{ForestGreen}{RGB}{34,139,34}
\definecolor{ao(english)}{rgb}{0.0, 0.5, 0.0}

\begin{document}

\title[Discrete Consensus-Based Optimization in Hilbert Spaces]
{Convergence of Time-Discrete Finite-Particle Consensus-Based Optimization in Hilbert Spaces}


\author{Michael Herty}
\address{Michael Herty \newline
Institut f\"ur Geometrie und Praktische Mathematik, RWTH Aachen University,  \newline
52062, Aachen, Germany. \newline
Extraordinary Professor, Department of Mathematics and Applied Mathematics, University of Pretoria, 
0028, South Africa}
\email{\texttt{herty@igpm.rwth-aachen.de}}

\author{Hui Huang}
\address{Hui Huang \newline
School of Mathematics, 
Hunan University, \newline
410082, Changsha, China
}
\email{\texttt{huihuang1@hnu.edu.cn}}

\author{Hicham Kouhkouh}
\address{Hicham Kouhkouh \newline 
Department of Mathematics and Scientific Computing, NAWI, University of Graz,\newline 
8010, Graz, Austria
}
\email{\texttt{hicham.kouhkouh@uni-graz.at}}





\date{\today}

\begin{abstract}
We study a time-discrete, finite-particle Consensus-Based Optimization (CBO) algorithm in a separable Hilbert space. Our analysis provides convergence guarantees directly for the computable particle system, complementing recent continuous-time and mean-field results in infinite dimensions. Using a common-noise formulation with trace-class covariance, we first establish quantitative pairwise contraction, exponential decay of the expected swarm variance, and almost-sure convergence of all particles to a common consensus state. We then combine estimates on the exponentiated objective functional with a quantitative Laplace principle to show that, for sufficiently large inverse temperature and suitably prepared initial data, the energy of the limiting consensus state can be made arbitrarily close to the global minimum over the active subspace. A key feature of the Hilbert-space formulation is that the stochastic contribution to the convergence estimates is controlled by the trace of the covariance operator and is therefore uniform with respect to the Galerkin dimension. Numerical experiments on an elliptic energy minimization problem with mixed boundary conditions and a PDE-constrained inverse source problem validate the theoretical convergence results and demonstrate stable performance under increasing spatial resolution, in contrast with CBO based on isotropic finite-dimensional noise.
\end{abstract}

\subjclass[MSC]{65C35, 60H10, 90C56, 90C26}








\keywords{
Consensus-Based Optimization,
Infinite-dimensional optimization,
Time-discrete particle systems,
Finite-particle convergence,
Hilbert spaces,
PDE-constrained optimization}

\maketitle


\section{Introduction}

In recent years, stochastic particle systems have emerged as a powerful tool for tackling large-scale and non-convex optimization problems arising in machine learning, physics, and applied mathematics. A central class of such methods is based on consensus-driven interaction mechanisms, which lead to collective dynamics. Therein, an ensemble of agents progressively concentrates around favorable regions of the energy landscape. Among these approaches, Consensus-Based Optimization (CBO) \cite{pinnau2017consensus,carrillo2018analytical} has become a prominent representative, motivated by collective behavior in biological and social systems. CBO employs interacting particle systems to explore the optimization landscape and has been extensively studied and extended in a variety of directions \cite{ha2022stochastic,borghi2022consensus,carrillo2022consensus,carrillo2021consensus,fornasier2022anisotropic,fornasier2020consensus,fornasier2024truncated,byeon2025consensus,wei2025consensus,cipriani2022zero,huang2024fast,huang2023global,roith2025consensus}. Its derivative-free nature has also made it particularly amenable to mathematical analysis, leading to a rapidly growing body of theoretical work on its convergence properties and mean-field behavior \cite{fornasier2024consensus,huang2025faithful,huang2025uniform,choi2025modified}.

Consequently, CBO has been developed and adapted to a wide array of complex optimization settings. Notable extensions include multiple-minimizer problems \cite{bungert2025polarized,fornasier2025pde,huang2025faithful}, saddle point problems \cite{huang2024consensus,borghi2024particle}, multiplayer games \cite{chenchene2025consensus}, neural networks \cite{de2025mean}, stochastic optimization \cite{bellavia2025discrete,bonandin2025consensus}, multi-level \cite{herty2025multiscale,garcia2025defending} and multi-objective optimization \cite{borghi2023adaptive,kang2026time}, clustered federated learning \cite{carrillo2023fedcbo}, as well as constrained optimization \cite{bungert2025mirrorcbo,beddrich2026constrained,borghi2023constrained}. Further developments encompass jump diffusions \cite{kalise2023consensus,aceves2026consensus}, momentum acceleration \cite{chen2022consensus}, discrete-time formulations \cite{ha2020convergence,ha2021convergence,ko2022convergence}, and memory-based or self-interacting dynamics \cite{riedl2024leveraging,huang2024self,totzeck2020consensus}. Beyond the classical Euclidean setting, consensus-driven optimization mechanisms have also recently been extended to non-Euclidean geometries, including Riemannian manifolds \cite{huang2026collective}. Rather than providing an exhaustive account of this rapidly expanding field, we refer to the survey by Totzeck \cite{totzeck2021trends} and the more recent comprehensive review \cite{fornasier2026consensus} for a broader perspective.

From the theoretical perspective, considerable effort has been devoted to understanding the relation between the finite-particle system and its mean-field description, as well as the long-time behavior of the resulting dynamics. Recent results include uniform-in-time mean-field estimates for CBO \cite{huang2025uniform,bayraktar2026uniform,gerber2025uniform}, long-time convergence for self-interacting CBO dynamics \cite{huang2024self} and for constant-diffusion CBO dynamics \cite{bianchi2025consensus}, and simultaneous investigations of consensus formation and uniform-in-time propagation of chaos for modified CBO models \cite{choi2025modified}. At the level of global optimization guarantees, increasingly refined convergence theories have also been developed, including faithful global convergence results designed to overcome limitations of classical CBO formulations \cite{huang2025faithful}. These advances provide a sharper understanding of the interplay between particle approximation, mean-field behavior, consensus formation, and global optimization.

\paragraph{\textbf{Time-discrete CBO}}
Discrete-time formulations of CBO are particularly relevant from the computational perspective, since they correspond directly to the algorithms implemented in practice. An important analytical foundation in this direction was established by Ha, Jin, and Kim. In \cite{ha2020convergence}, the authors study the convergence of a first-order consensus-based global optimization algorithm, while \cite{ha2021convergence} develops convergence and error estimates for time-discrete CBO schemes. These works provide a direct finite-particle analysis of the discrete dynamics and establish a useful framework for separating the formation of consensus from the characterization of the limiting consensus state. More recent developments have addressed discrete CBO in realistic algorithmic settings, including noisy objective evaluations \cite{bellavia2025discrete} and strong mean-square convergence of practical time-discrete schemes with quantitative dependence on the time step and the number of particles \cite{bonandin2025strong}. These results establish a convergence theory for implementable CBO algorithms in finite-dimensional Euclidean spaces.

\paragraph{\textbf{CBO in Hilbert spaces.}}
While the finite-dimensional setting has been extensively studied, many optimization problems arising in modern applications are inherently infinite-dimensional. Representative examples include variational problems, optimal control, inverse problems, and optimization constrained by partial differential equations. Such problems are naturally formulated over function spaces, and in particular over separable Hilbert spaces. Extending particle-based optimization methods such as CBO to these settings is therefore both natural and necessary, but introduces fundamental analytical and computational difficulties associated with the infinite-dimensional geometry and with the construction of admissible stochastic perturbations.

Recent works have approached infinite-dimensional optimization with CBO from several complementary perspectives. For instance, \cite{khatab2026consensus} employs Gaussian process representations within a CBO framework for optimization in Sobolev spaces, whereas \cite{borghi2026variational} formulates CBO on the manifold of Gaussian probability measures endowed with a Bures--Wasserstein geometry for variational inference. CBO has also been applied to problems whose original formulation is infinite-dimensional while the numerical optimization itself is performed over a finite-dimensional parametrization, including stochastic control problems \cite{lyu2025consensus} and model predictive control formulations \cite{borghi2025model}.

More recently, in \cite{huang2026derivative}, CBO was formulated directly in a separable Hilbert space at the level of continuous-time stochastic dynamics and their mean-field McKean--Vlasov description. A central difficulty in passing from Euclidean to Hilbert spaces concerns the construction of the stochastic exploration mechanism. While isotropic Brownian forcing is natural in finite dimensions, its direct infinite-dimensional analogue is generally not Hilbert-space valued. The formulation in \cite{huang2026derivative} therefore relies on trace-class covariance operators and $Q$-Wiener processes, together with an active subspace structure identifying the directions on which stochastic exploration acts. Within this framework, well-posedness and quantitative concentration properties of the continuous-time and mean-field dynamics were established.

The continuous-time and mean-field theory, however, does not by itself provide convergence guarantees for the numerical CBO algorithm. In practice, one evolves a finite ensemble of particles with a nonzero time step and, for function-space optimization, represents the particles in a finite-dimensional approximation space. This raises a distinct question: whether the consensus and global optimization mechanisms persist simultaneously under time discretization, a finite number of particles, and refinement of the underlying spatial discretization. In particular, an infinite-dimensional algorithmic formulation should yield estimates that remain controlled as the dimension of the active approximation space increases. Establishing convergence directly for the spatially truncated, time-discrete particle system is therefore essential for connecting the Hilbert-space theory of CBO with its numerical implementation.

\paragraph{\textbf{The present work.}}
In this manuscript, we address this problem by establishing a convergence theory for a time-discrete, finite-particle CBO algorithm in a separable Hilbert space $H$ without resorting to the mean field approximation. Building on the discrete finite-particle strategy developed in \cite{ha2020convergence,ha2021convergence} and the infinite-dimensional stochastic framework of \cite{huang2026derivative}, we consider the optimization problem
\begin{equation*}
    \text{Find } x_V^* \in \argmin_{x\in V}\mathscr{E}(x),
\end{equation*}
where $V\subset H$ denotes a finite-dimensional active subspace and $\mathscr{E}:H\to\mathbb{R}$ is a possibly non-convex objective functional. Our analysis proceeds in two main steps. We first establish quantitative contraction of the finite-particle system and prove the emergence of global consensus. We then investigate the optimization quality of the limiting consensus state and show, through a quantitative Laplace principle, that its energy can be made arbitrarily close to the global minimum of $\mathscr{E}$ over the active subspace.

\paragraph{\textbf{Convergence analysis.}}
More precisely, we formulate the computable dynamics on an $M$-dimensional active subspace $V=\operatorname{span}\{\phi_1,\ldots,\phi_M\}\subset H$ and drive all particles at a given time step by the same realization of the truncated trace-class noise
\begin{equation*}
    Z_n^Q := \sum_{k=1}^M \sqrt{\beta_k}\,\,\zeta_{n,k}\phi_k,
\end{equation*}
where $(\beta_k)_{k\geq 1}$ are the eigenvalues of a trace-class covariance operator $Q$ and $\{\zeta_{n,k}\}_{k=1}^M$ are independent standard normal random variables. The use of a common noise realization is fundamental to the pairwise analysis. Indeed, subtracting the dynamics of two particles $i$ and $j$ gives
\begin{equation*}
\begin{aligned}
    X_{n+1}^i-X_{n+1}^j ={}&(1-\lambda h)(X_n^i-X_n^j) +\sigma\sqrt{h}\left(\|X_n^i-\mathfrak{m}_n^{\alpha,N}\|_H-\|X_n^j-\mathfrak{m}_n^{\alpha,N}\|_H\right)Z_n^Q,
\end{aligned}
\end{equation*}
where $h = \Delta t > 0$ denotes the constant time step, $\mathfrak{m}_n^{\alpha,N}$ is the empirical consensus point at step $n$, $N$ is the number of particles, and $\lambda,\sigma>0$ are fixed constants. 
Hence, although the stochastic contribution does not disappear entirely from the difference dynamics, the same random direction acts on every particle, while the dependence on the empirical consensus point can be controlled through the reverse triangle inequality
\begin{equation*}
    \left|\|X_n^i-\mathfrak{m}_n^{\alpha,N}\|_H-\|X_n^j-\mathfrak{m}_n^{\alpha,N}\|_H\right| \leq \|X_n^i-X_n^j\|_H.
\end{equation*}
Together with the conditional mean-zero property of the noise and the identity
\begin{equation*}
    \mathbb{E}\left[\|Z_n^Q\|_H^2\right]=C_{Q,M}:=\sum_{k=1}^M\beta_k,
\end{equation*}
this provides the basic contraction mechanism underlying the emergence of consensus.

In particular, under a suitable relation between the algorithmic parameters, we prove exponential decay of the expected swarm variance
\begin{equation*}
    V_n:=\frac{1}{N}\sum_{i=1}^N\mathbb{E}\left[\|X_n^i-\mathfrak{m}_n^{\alpha,N}\|_H^2\right].
\end{equation*}
More precisely, we obtain, for some constant $m>0$,
\begin{equation*}
    V_n\leq N(1-hm)^nD_{\mathrm{pair}},
    \qquad \text{where} \qquad
    D_{\mathrm{pair}}:=\max_{1\leq i,j\leq N}\mathbb{E}\left[\|X_0^i-X_0^j\|_H^2\right].
\end{equation*}
We further establish the existence of a common random state $X_\infty\in V$ such that, for every $i=1,\ldots,N$,
\begin{equation*}
    X_n^i\longrightarrow X_\infty \qquad \text{almost surely as } n\to\infty.
\end{equation*}
Thus, the discrete finite-particle system exhibits global consensus, with an exponentially collapsing swarm variance.

Consensus alone, however, does not guarantee that the limiting state is located near a global minimizer. The second part of our analysis therefore concerns the optimization quality of $X_\infty$. We study the evolution of the expected exponentiated energy of the swarm and combine the resulting estimates with a quantitative Laplace principle. Assuming suitable regularity of the objective functional and a well-prepared initial distribution assigning positive probability to arbitrary energy neighborhoods of $x_V^*$, we show that, for every prescribed accuracy $\eta>0$ and any fixed $\varepsilon\in(0,1)$, the inverse temperature $\alpha$ can be chosen sufficiently large such that
\begin{equation*}
    -\frac{\log\varepsilon}{\alpha}-\frac{1}{\alpha}\log\left(\mathbb{E}\left[e^{-\alpha(\mathscr{E}(X_{in})-\mathscr{E}_{\min})}\right]\right)<\eta.
\end{equation*}
Provided that the initial swarm satisfies the corresponding well-preparedness condition at this value of $\alpha$, the limiting consensus state then fulfills
\begin{equation*}
    0\leq \essinf\limits_{\omega\in\Omega}\mathscr{E}(X_\infty)-\mathscr{E}_{\min}<\eta.
\end{equation*}
The analysis therefore separates two fundamental mechanisms underlying the algorithm: the exponential contraction of the particle system toward a common consensus state and the selection, through Gibbs reweighting and stochastic exploration, of a consensus state whose energy can be made arbitrarily close to the global minimum.

\paragraph{\textbf{Dimension-robust exploration.}}
A further feature of the Hilbert-space formulation is that the same stochastic structure responsible for well-defined infinite-dimensional exploration naturally yields stability with respect to the Galerkin dimension. Indeed, the truncated covariance satisfies
\begin{equation*}
    0<C_{Q,M}=\sum_{k=1}^M\beta_k\leq \operatorname{Trace}(Q)=\sum_{k=1}^{\infty}\beta_k<\infty,
\end{equation*}
uniformly in $M$. This is in sharp contrast with standard isotropic exploration in an $M$-dimensional Euclidean approximation, for which the total stochastic variance grows proportionally to the ambient dimension. Consequently, increasing the Galerkin dimension in order to resolve an underlying PDE or variational problem more accurately may simultaneously degrade the behavior of an isotropic CBO algorithm by injecting increasing amounts of stochastic energy into high-frequency directions. Trace-class exploration instead distributes a finite amount of stochastic energy across the modes of the Hilbert space, with decreasing excitation of higher-frequency components. Thus, a structural requirement originating from infinite-dimensional stochastic analysis leads naturally to a dimension-robust exploration mechanism at the computational level. This perspective is complementary to recent developments emphasizing anisotropic and structure-aware diffusion as a means of mitigating ambient-dimensional effects in high-dimensional CBO \cite{bonandin2026exploiting}.

\subsection{Main Contributions}

The main contributions of this work can be summarized as follows.
\begin{enumerate}
    \item \textbf{Discrete finite-particle consensus with dimension-uniform control.} We extend the discrete finite-particle convergence strategy of \cite{ha2020convergence,ha2021convergence} to the Hilbert-space setting. By exploiting the common-noise structure of the particle dynamics, we derive quantitative pairwise contraction estimates whose stochastic contribution is controlled by $C_{Q,M}$. Since $C_{Q,M}\leq\operatorname{Trace}(Q)$, the resulting stability condition can be controlled independently of the Galerkin dimension $M$.

    \item \textbf{Exponential collapse of the swarm and almost-sure consensus.} We prove exponential decay of the expected swarm variance and establish the existence of a common random state $X_\infty\in V$ such that all particles converge to $X_\infty$ almost surely. This provides a convergence result directly for the finite-particle, time-discrete dynamics used in computation, rather than for an associated continuous-time or mean-field limit.

    \item \textbf{Global optimization guarantee via a quantitative Laplace principle.} We complement the consensus analysis by controlling the expected exponentiated energy of the swarm. Under suitable assumptions on $\mathscr{E}$ and a well-prepared initial distribution, a quantitative Laplace principle shows that the energy of the limiting consensus state can be made arbitrarily close to the active-subspace global minimum by choosing the inverse temperature $\alpha$ sufficiently large and the initial swarm sufficiently concentrated in the sense required by the analysis.

    \item \textbf{Dimension-robust exploration and numerical validation.} The trace-class covariance structure yields a uniform bound on the total stochastic variance as the active dimension increases, in contrast with isotropic finite-dimensional noise. We illustrate the performance of the proposed Hilbert-space CBO scheme through two infinite-dimensional optimization problems: an elliptic energy minimization problem with mixed boundary conditions and a PDE-constrained inverse source problem. The numerical experiments confirm the exponential collapse of the swarm and convergence of the objective values. Moreover, comparisons across increasing Galerkin dimensions demonstrate the dimension-robust behavior of the Hilbert-space CBO scheme and highlight the deterioration of its isotropic finite-dimensional counterpart as the spatial resolution increases.
\end{enumerate}

These results provide a direct link between the continuous-time and mean-field Hilbert-space framework of \cite{huang2026derivative} and the finite-particle algorithms used in practice. In particular, they show that the trace-class stochastic structure required by the infinite-dimensional formulation is not only analytically admissible, but also leads to a natural mechanism for controlling the discrete optimization dynamics under refinement of the spatial approximation.


\subsection{Organization of the paper}
The remainder of the manuscript is organized as follows. \textbf{Section \ref{sec:discrete_system}} introduces the time-discrete finite-particle CBO system and the underlying Hilbert-space framework. \textbf{Section \ref{sec:consensus}} establishes the emergence of global consensus, including the exponential decay of the swarm variance and the almost-sure convergence of all particles to a common limiting state. \textbf{Section \ref{sec:global_convergence}} investigates the optimization quality of this consensus state and establishes convergence toward the global minimum over the active subspace using a quantitative Laplace principle. Finally, \textbf{Section \ref{sec:numerics}} presents numerical illustrations for an elliptic energy minimization problem with mixed boundary conditions and a PDE-constrained inverse source problem, with particular emphasis on convergence behavior and robustness with respect to the Galerkin dimension.

\section{The Time-Discrete Particle System}\label{sec:discrete_system}

Let $h = \Delta t > 0$ denote the constant time step, and let $X_n^i \in V \subset H$ denote the state of the $i$-th particle at time $t = nh$, and where $V$ is a finite-dimensional subspace. To ensure our convergence guarantees directly apply to the computable algorithm, we formulate our analysis on the spatially-truncated (Galerkin) discrete-time CBO scheme. This scheme is driven by a \textbf{common} set of simulated standard normal increments.

\begin{definition}[Discrete-Time CBO Scheme]\label{def}
For $i = 1, \dots, N$, the particle update is given by the recursive relation
\begin{equation}\label{eq:dynamics recursive}
    X_{n+1}^i = X_n^i - \lambda h (X_n^i - \mathfrak{m}_n^{\alpha, N}) + \sigma \sqrt{h} \|X_n^i - \mathfrak{m}_n^{\alpha, N}\|_H \left( \sum_{k=1}^M \sqrt{\beta_k}\, \zeta_{n,k} \phi_k \right)
\end{equation}
where
\begin{itemize}
    \item $\mathfrak{m}_n^{\alpha, N}$ is the empirical consensus point at step $n$, defined as:
    \begin{equation}
        \mathfrak{m}_n^{\alpha, N} = \sum_{j=1}^N \omega_n^j X_n^j, \quad \text{with } \omega_n^j = \frac{\exp(-\alpha \mathscr{E}(X_n^j))}{\sum_{\ell=1}^N \exp(-\alpha \mathscr{E}(X_n^\ell))}
    \end{equation}
    \item $\{\phi_k\}_{k=1}^M$ forms an orthonormal basis for the $M$-dimensional active subspace $V$. Because the stochastic sum is constructed exclusively from this basis, it resides in $V$, superseding the need for an explicit projection operator $\mathsf{P}_V$.
    \item $\beta_k > 0$ are the  eigenvalues of the trace-class covariance operator $Q$.
    \item $\{\zeta_{n,k}\}_{k=1}^M$ is a set of independent standard normal random variables, i.e., $\zeta_{n,k} \sim \mathcal{N}(0,1)$, shared across all $N$ particles at step $n$.
    \item $\mathcal{F}_n$ is the discrete-time filtration up to step $n$. Given the independence and unit variance of the normal draws, the stochastic term is mean-zero given $\mathcal{F}_n$, and its expected squared norm evaluates proportionally to the truncated trace constant $C_{Q,M} := \sum_{k=1}^M \beta_k$. Moreover $0< C_{Q,M}\leq \mathrm{Trace}(Q)=\sum_{k=1}^\infty \beta_k<\infty$.
\end{itemize}
\end{definition}


\begin{remark}[The Noise Increment]\label{rmk: noise}
The stochastic term 
\[
\mathcal{N}_n^i := \sigma \sqrt{h} \|X_n^i - \mathfrak{m}_n^{\alpha, N}\|_H \left( \sum_{k=1}^M \sqrt{\beta_k}\, \zeta_{n,k} \phi_k \right)
\]
satisfies the following properties.
\begin{enumerate}
    \item \textbf{Conditional Mean-Zero.} Since $X_n^i$ and $\mathfrak{m}_n^{\alpha, N}$ are $\mathcal{F}_n$-measurable, and the normal draws are independent of $\mathcal{F}_n$ with $\mathbb{E}[\zeta_{n,k}] = 0$, we have
    \begin{equation*}
        \mathbb{E}[\mathcal{N}_n^i \mid \mathcal{F}_n] = \sigma \sqrt{h} \|X_n^i - \mathfrak{m}_n^{\alpha, N}\|_H \sum_{k=1}^M \sqrt{\beta_k}\, \mathbb{E}[\zeta_{n,k}] \phi_k = 0.
    \end{equation*}
    \item \textbf{Conditional Variance.} using the orthonormality of the basis $\{\phi_k\}$ in $H$ and the variance property $\mathbb{E}[\zeta_{n,k}^2] = 1$, one gets
    \begin{align*}
        \mathbb{E}[\|\mathcal{N}_n^i\|_H^2 \mid \mathcal{F}_n] &= \sigma^2 h \|X_n^i - \mathfrak{m}_n^{\alpha, N}\|_H^2 \mathbb{E} \left[ \left\| \sum_{k=1}^M \sqrt{\beta_k}\, \zeta_{n,k} \phi_k \right\|_H^2 \right]  \\
        &= \sigma^2 h \|X_n^i - \mathfrak{m}_n^{\alpha, N}\|_H^2 \sum_{k=1}^M \beta_k \mathbb{E}[\zeta_{n,k}^2] 
        = \sigma^2 h C_{Q,M} \|X_n^i - \mathfrak{m}_n^{\alpha, N}\|_H^2,
    \end{align*}
    where $C_{Q,M} := \sum_{k=1}^M \beta_k$ is the truncated trace constant.
\end{enumerate}
\end{remark}

\section{Emergence of Global Consensus}\label{sec:consensus}

To establish that the particle swarm collapses to a single consensus point, we analyze the evolution of the squared Hilbert distance between any two arbitrary particles $i$ and $j$. By driving the system with a common noise realization at each time step, the stochastic perturbations acting on the particle positions cancel out in the difference, allowing us to derive sharp exponential decay guarantees.

\begin{theorem}[
Consensus, Part I]\label{thm:strong_consensus}
Let $C_{Q,M} := \sum_{k=1}^M \beta_k$ denote the trace of the truncated covariance operator. Suppose the system parameters satisfy $2\lambda > \sigma^2 C_{Q,M}$ and the time step $h$ is chosen such that 
\[
0 < h < \frac{2\lambda - \sigma^2 C_{Q,M}}{\lambda^2} 
\]
Then, for any pair of particles $i, j \in \{1, \dots, N\}$, the expected relative distance satisfies 
\begin{equation*}
    \mathbb{E}[\,\|X_n^i - X_n^j\|_H^2\,] \;\le\; (1 - hm)^n\, \mathbb{E}[\,\|X_0^i - X_0^j\|_H^2\,]
\end{equation*}
where $m := 2\lambda - \lambda^2 h - \sigma^2 C_{Q,M} > 0$. Consequently, for $h$ small such that $0<hm<1$, this difference decays exponentially, and
the particles achieve consensus in the mean-square sense $\lim_{n \to \infty} \mathbb{E}[\|X_n^i - X_n^j\|_H^2] = 0$ for all $i,j$.
\end{theorem}

\begin{remark}\label{rmk:independent of M}
Recalling $0< C_{Q,M}\leq \mathrm{Trace}(Q)=\sum_{k=1}^\infty \beta_k<\infty$, the condition on $h$ in Theorem \ref{thm:strong_consensus} can be made independent from the dimension $M$ of the active subspace, such that
\[
0< h < \frac{2\lambda - \sigma^2 \mathrm{Trace}(Q)}{\lambda^2}
\qquad \text{ and } \qquad 
m = 2\lambda - \lambda^2 h - \sigma^2 \mathrm{Trace}(Q) > 0.
\]
\end{remark}

\begin{proof}
Let $X_n^{i,j} := X_n^i - X_n^j$. Subtracting the equations for particles $i$ and $j$ yields
\begin{equation*}
    X_{n+1}^{i,j} = (1 - \lambda h)X_n^{i,j} + \sigma \sqrt{h} \left( \|X_n^i - \mathfrak{m}_n^{\alpha, N}\|_H - \|X_n^j - \mathfrak{m}_n^{\alpha, N}\|_H \right) \left( \sum_{k=1}^M \sqrt{\beta_k}\, \zeta_{n,k} \phi_k \right).
\end{equation*}
Taking the squared Hilbert norm on both sides and conditioning on the filtration $\mathcal{F}_n$, the cross-term between the deterministic part and the stochastic noise vanishes since $\mathbb{E}[\zeta_{n,k} \mid \mathcal{F}_n] = 0$ and the other terms are $\mathcal{F}_n$-measurable. Thus one obtains
\begin{equation}\label{eq:consensus_expansion}
\begin{aligned}
    & \mathbb{E}[\|X_{n+1}^{i,j}\|_H^2 \mid \mathcal{F}_n] = (1 - \lambda h)^2 \|X_n^{i,j}\|_H^2 \\
    &\quad\qquad + \sigma^2 h \left( \|X_n^i - \mathfrak{m}_n^{\alpha, N}\|_H - \|X_n^j - \mathfrak{m}_n^{\alpha, N}\|_H \right)^2 \mathbb{E}\left[ \big\| \sum_{k=1}^M \sqrt{\beta_k}\, \zeta_{n,k} \phi_k \big\|_H^2 \,\middle|\, \mathcal{F}_n \right]. 
\end{aligned}
\end{equation}

Using the orthonormality of the basis $\{\phi_k\}_{k=1}^M$, the squared norm of the stochastic sum simplifies via Parseval's identity such that
\begin{equation*}
    \mathbb{E}\left[ \big\| \sum_{k=1}^M \sqrt{\beta_k}\, \zeta_{n,k} \phi_k \big\|_H^2 \right] = \sum_{k=1}^M \beta_k \mathbb{E}[\zeta_{n,k}^2] = \sum_{k=1}^M \beta_k = C_{Q,M}.
\end{equation*}

Furthermore, by the reverse triangle inequality, the difference of the norms is bounded
\begin{equation*}
    \left| \|X_n^i - \mathfrak{m}_n^{\alpha, N}\|_H - \|X_n^j - \mathfrak{m}_n^{\alpha, N}\|_H \right|^2 \le \| (X_n^i - \mathfrak{m}_n^{\alpha, N}) - (X_n^j - \mathfrak{m}_n^{\alpha, N}) \|_H^2 = \|X_n^{i,j}\|_H^2.
\end{equation*}

Substituting these into \eqref{eq:consensus_expansion}, we obtain
\begin{align*}
    \mathbb{E}[\|X_{n+1}^{i,j}\|_H^2 \mid \mathcal{F}_n] &\le (1 - \lambda h)^2 \|X_n^{i,j}\|_H^2 + \sigma^2 h C_{Q,M} \|X_n^{i,j}\|_H^2 \\
    &= \left( 1 - 2\lambda h + \lambda^2 h^2 + \sigma^2 h C_{Q,M} \right) \|X_n^{i,j}\|_H^2 \\
    &= \left( 1 - h(2\lambda - \lambda^2 h - \sigma^2 C_{Q,M}) \right) \|X_n^{i,j}\|_H^2 \,
    = (1 - hm) \|X_n^{i,j}\|_H^2.
\end{align*}

Given the parameter conditions $2\lambda > \sigma^2 C_{Q,M}$ and $h < \frac{2\lambda - \sigma^2 C_{Q,M}}{\lambda^2}$, we have $m > 0$.
Applying the law of total expectation and iterating over $n$ steps yields
\begin{equation*}
    \mathbb{E}[\|X_n^{i,j}\|_H^2] \le (1 - hm)^n \mathbb{E}[\|X_0^{i,j}\|_H^2].
\end{equation*}
When $0<hm<1$, i.e. $0 < 1 - hm < 1$, the expected squared distance vanishes exponentially as $n \to \infty$.
\end{proof}

\begin{corollary}
[Variance Decay]
\label{cor:variance_decay}
In the situation of Theorem \ref{thm:strong_consensus}, the expected  variance of the swarm $V_n := \frac{1}{N} \sum_{i=1}^N \mathbb{E}[\|X_n^i - \mathfrak{m}_n^{\alpha, N}\|_H^2]$ decays exponentially such that
\begin{equation*}
    V_n \le N (1 - hm)^n D_{\text{pair}},
\end{equation*}
where $D_{\text{pair}} := \max_{1 \le i, j \le N} \mathbb{E}[\|X_0^i - X_0^j\|_H^2]$ is the maximum expected initial pairwise squared distance.
\end{corollary}

\begin{proof}
Recall the consensus point $\mathfrak{m}_n^{\alpha, N} = \sum_{j=1}^N \omega_n^j X_n^j$. By the convexity of the squared Hilbert norm, we bound the distance of any particle $i$ to the consensus point via Jensen's inequality
\begin{equation*}
    \|X_n^i - \mathfrak{m}_n^{\alpha, N}\|_H^2 = \left\| \sum_{j=1}^N \omega_n^j (X_n^i - X_n^j) \right\|_H^2 \le \sum_{j=1}^N \omega_n^j \|X_n^i - X_n^j\|_H^2.
\end{equation*}
Because the Gibbs weights are bounded ($0<\omega_n^j \le 1$), we drop them to recover a bound independent of $\alpha$ such as
\begin{equation*}
    \|X_n^i - \mathfrak{m}_n^{\alpha, N}\|_H^2 \le \sum_{j=1}^N \|X_n^i - X_n^j\|_H^2.
\end{equation*}

We can now take the unconditional expectation, then we average over all particles $i = 1, \dots, N$, and we substitute the decay bound from Theorem \ref{thm:strong_consensus}, obtaining
\begin{align*}
    V_n &\le \frac{1}{N} \sum_{i=1}^N \sum_{j=1}^N \mathbb{E}[\|X_n^i - X_n^j\|_H^2] \\
    &\le \frac{1}{N} \sum_{i=1}^N \sum_{j=1}^N (1 - hm)^n \mathbb{E}[\|X_0^i - X_0^j\|_H^2].
\end{align*}
Bounding each pairwise expectation by the maximum initial variance $D_{\text{pair}}$ gives $N^2$ identical terms. Dividing by the $1/N$ prefactor yields the  bound $N(1-hm)^n D_{\text{pair}}$.
\end{proof}

\begin{theorem}[Consensus, Part II]\label{thm:common_state}
In the situation of Theorem \ref{thm:strong_consensus}, there exists a common constant state $X_\infty \in V$ such that for every particle $i = 1, \dots, N$,
\begin{equation}
    \lim_{n \to \infty} X_n^i = X_\infty \quad \text{almost surely.}
\end{equation}
\end{theorem}

\begin{proof}
We shall consider the term-by-term sum of \eqref{eq:dynamics recursive} over the time indices. 
Let the stochastic increment be denoted $\mathcal{N}_p^i$. By expanding the discrete update rule recursively from $p=0$, one gets
\begin{equation*}
    X_n^i = X_0^i - \lambda h \sum_{p=0}^{n-1} (X_p^i - \mathfrak{m}_p^{\alpha, N}) + \sum_{p=0}^{n-1} \mathcal{N}_p^i =: X_0^i - \mathcal{S}_n^i + \mathcal{M}_n^i.
\end{equation*}

We analyze the almost sure convergence of the deterministic sum $\mathcal{S}_n^i$ and the stochastic sum $\mathcal{M}_n^i$ separately.

\textbf{Step 1.} \textit{(The martingale sum $\mathcal{M}_n^i$.)} \\
By Remark \ref{rmk: noise}, the stochastic increments $\mathcal{N}_p^i$ are mean-zero given the filtration $\mathcal{F}_p$. Thus, $\mathcal{M}_n^i$ is an $H$-valued discrete-time martingale. We evaluate its variance in $L^2(\Omega; H)$ that is
\begin{equation*}
    \mathbb{E}[\|\mathcal{M}_n^i\|_H^2] = \sum_{p=0}^{n-1} \mathbb{E}[\|\mathcal{N}_p^i\|_H^2] = \sigma^2 h C_{Q,M} \sum_{p=0}^{n-1} \mathbb{E}[\|X_p^i - \mathfrak{m}_p^{\alpha, N}\|_H^2].
\end{equation*}
Corollary \ref{cor:variance_decay} ensures that the expected swarm variance $V_p$ decays such that $V_p \le (1-hm)^p ND_{\text{pair}}$. Bounding the individual distance by the total swarm variance $N V_p$, we have
\begin{equation*}
    \mathbb{E}[\|\mathcal{M}_n^i\|_H^2] \le \sigma^2 h C_{Q,M} N^2D_{\text{pair}} \sum_{p=0}^{n-1} (1-hm)^p.
\end{equation*}
Because $1-hm < 1$, the infinite geometric series converges. Consequently, we have $\sup_n \mathbb{E}[\|\mathcal{M}_n^i\|_H^2] < \infty$. Since the martingale is bounded in $L^2$, Doob's Martingale Convergence Theorem for Hilbert spaces \cite[Chapter IV, Theorem 3]{scalora1958abstract} guarantees that the limit $\lim_{n \to \infty} \mathcal{M}_n^i$ exists almost surely.

\textbf{Step 2.} \textit{(The deterministic sum $\mathcal{S}_n^i$.)} \\
To prove that $\mathcal{S}_n^i$ converges, we show that the series is absolutely convergent almost surely. By Jensen's inequality for expectations and the swarm variance bound, it holds
\begin{equation*}
    \mathbb{E}[\|X_p^i - \mathfrak{m}_p^{\alpha, N}\|_H] \le \sqrt{\mathbb{E}[\|X_p^i - \mathfrak{m}_p^{\alpha, N}\|_H^2]} \le \sqrt{N^2D_{\text{pair}}} \big( \sqrt{1-hm} \big)^p.
\end{equation*}
Summing this expected norm from $p=0$ to $\infty$ yields a geometric series with a common ratio $\sqrt{1-hm} < 1$, thus
\begin{equation*}
    \sum_{p=0}^\infty \mathbb{E}[\|X_p^i - \mathfrak{m}_p^{\alpha, N}\|_H] < \infty.
\end{equation*}
By the Monotone Convergence Theorem (or Fubini's Theorem for non-negative random variables), we may exchange the expectation and the infinite sum:
\begin{equation*}
    \mathbb{E} \left[ \sum_{p=0}^\infty \|X_p^i - \mathfrak{m}_p^{\alpha, N}\|_H \right] < \infty.
\end{equation*}
Because the expectation of the infinite sum is finite, the sum itself must be finite almost surely. By definition, a Hilbert space $H$ is a complete metric space. A fundamental property of complete spaces is that absolute convergence implies convergence, therefore the sequence of partial sums forms a Cauchy sequence. Consequently, the deterministic series converges to a well-defined limit vector $\mathcal{S}_\infty^i \in H$ almost surely.

\textbf{Step 3.} \textit{(Emergence of a common state.)} \\
Recall the recursive expansion of the particle state: $X_n^i = X_0^i - \mathcal{S}_n^i + \mathcal{M}_n^i$. Because both the deterministic sum $\mathcal{S}_n^i$ and the martingale sum $\mathcal{M}_n^i$ converge almost surely to finite limits ($\mathcal{S}_\infty^i$ and $\mathcal{M}_\infty^i$, respectively), their linear combination must also converge. Thus, the sequence $X_n^i$ converges almost surely to some  random variable that is $X_\infty^i := X_0^i - \mathcal{S}_\infty^i + \mathcal{M}_\infty^i \in H$.

To prove that all particles converge to this same state, we recall from Theorem \ref{thm:strong_consensus} that $\mathbb{E}[\|X_n^i - X_n^j\|_H^2] \le (1-hm)^n N^2D_{\text{pair}}$. Applying Jensen's inequality and summing over time yields a convergent geometric series:
\begin{equation*}
    \sum_{n=0}^\infty \mathbb{E}[\|X_n^i - X_n^j\|_H] \le \sqrt{N^2D_{\text{pair}}} \sum_{n=0}^\infty \big(\sqrt{1-hm}\big)^n < \infty.
\end{equation*}
By the Monotone Convergence Theorem, $\mathbb{E} \left[ \sum_{n=0}^\infty \|X_n^i - X_n^j\|_H \right] < \infty$, which implies the infinite sum is finite almost surely. Since the sum converges, the individual terms must vanish, i.e. $\lim_{n \to \infty} \|X_n^i - X_n^j\|_H = 0$ almost surely. Therefore, for all $i, j \in \{1, \dots, N\}$, the individual limits must coincide: $X_\infty^i = X_\infty^j =: X_\infty$ almost surely. The proof shows that $X_{\infty}\in H$, but since $X^{i}_{n}$ remains in $V$ (see Definition \ref{def}) which is a subspace of $H$ hence is closed, one gets $X_{\infty}\in V$.
\end{proof}



\section{Global Convergence to the Minimum}\label{sec:global_convergence}

Having established that the particle swarm collapses to a common consensus point, we now track the evolution of the expected exponentiated energy of the swarm to prove convergence toward the active subspace minimizer $x_V^*$.

 We define the framework for the objective function and the reference initial data.

\begin{assumption}
\label{assump:error_framework} The objective function $\mathscr{E}$  and the initial data satisfy the following.
\begin{enumerate}

    \item The objective function $\mathscr{E}: H \to \mathbb{R}$ is continuously Fr\'echet differentiable ($C^1$) and attains a global minimum $\mathscr{E}_{\min} := \mathscr{E}(x_V^*) = \min_{x \in V} \mathscr{E}(x) \ge 0$ for some $x_V^* \in V$. Furthermore, it satisfies the following two conditions for all $x, y \in H$:
    \begin{itemize}
        \item \textit{One-Sided Lipschitz (semiconvexity):}  $\exists\,C_{\mathscr{E}} > 0$ a constant such that
        \begin{equation}\label{eq:osl_condition}
            \langle \nabla \mathscr{E}(x) - \nabla \mathscr{E}(y), x - y \rangle_H \ge -C_{\mathscr{E}} \|x - y\|_H^2.
        \end{equation}
        \item \textit{Quadratic Upper-Bound (semiconcavity):} $\exists\,L_{\mathscr{E}} > 0$ a constant such that
        \begin{equation}\label{eq:quadratic_condition}
            \mathscr{E}(y) - \mathscr{E}(x) \le \langle \nabla \mathscr{E}(x), y - x \rangle_H + \frac{L_{\mathscr{E}}}{2} \|y - x\|_H^2.
        \end{equation}
    \end{itemize}
    \item 
    The initial positions $\{X_0^i\}_{i=1}^N$ are drawn i.i.d. from a reference random variable $X_{in} \in V$. The law of $X_{in}$ has a finite second moment ($\mathbb{E}[\|X_{in}\|_H^2] < \infty$) and is well-prepared, meaning it assigns strictly positive probability mass to any arbitrary energy neighborhood of the active subspace minimizer $x_V^*$.
\end{enumerate}
\end{assumption}

\begin{remark}\label{rmk:interpretation}
The estimate \eqref{eq:quadratic_condition} states that the first-order Taylor approximation of $\mathscr{E}$ at $x$ provides a global quadratic upper bound on the energy. Moreover, interchanging the roles of $x$ and $y$ and adding the resulting inequalities yields
\[
\langle \nabla\mathscr{E}(x)-\nabla\mathscr{E}(y), x-y \rangle_H \le L_{\mathscr{E}}\|x-y\|_H^2,
\]
which may be interpreted as a one-sided upper Lipschitz bound on the gradient. Equivalently, $\mathscr{E}$ is $L_{\mathscr{E}}$-semiconcave, i.e., $x\mapsto \mathscr{E}(x)-\frac{L_{\mathscr{E}}}{2}\|x\|_H^2$ is concave.

Similarly, the assumption \eqref{eq:osl_condition} is a one-sided Lipschitz (or hypomonotonicity) condition on the gradient. Equivalently, $\mathscr{E}$ is $C_{\mathscr{E}}$-semiconvex, i.e., $x\mapsto \mathscr{E}(x)+\frac{C_{\mathscr{E}}}{2}\|x\|_H^2$ is convex. 

When together, they provide two-sided control on the directional variation of the gradient
\[
-C_{\mathscr{E}}\|x-y\|_H^2 \le \langle \nabla\mathscr{E}(x)-\nabla\mathscr{E}(y), x-y \rangle_H \le L_{\mathscr{E}}\|x-y\|_H^2.
\]
Although each of the two inequalities separately provides only a one-sided directional control on the variation of the gradient, their combination implies the global Lipschitz estimate
\[
\|\nabla\mathscr{E}(x)-\nabla\mathscr{E}(y)\|_H \le \bigl(L_{\mathscr{E}}+C_{\mathscr{E}}\bigr)\|x-y\|_H.
\]
Consequently, $\mathscr{E}\in C^{1,1}(H)$. See Appendix \ref{app:C11}. 

If moreover $\mathscr{E}$ is twice Fr\'echet differentiable, these assumptions are respectively equivalent to
\[
\nabla^2\mathscr{E}(x)\preceq L_{\mathscr{E}}I, \qquad \nabla^2\mathscr{E}(x)\succeq -C_{\mathscr{E}}I,
\]
where $\preceq$ denotes the Loewner order on self-adjoint operators. Thus, they may be viewed as first-order counterparts of one-sided Hessian bounds, while remaining meaningful without explicitly assuming the existence of second derivatives.
\end{remark}

For brevity, we define the dimensionless drift and noise variance parameters over a single time step
\begin{equation*}
    \gamma := \lambda h, \quad \text{and} \quad \zeta^2 := \sigma^2 h C_{Q,M}.
\end{equation*}
Recalling Theorem \ref{thm:strong_consensus}, we require for convergence that $m = 2\lambda - \lambda^2 h - \sigma^2 C_{Q,M} > 0$, and $1-hm<1$, which translate to $(1-\gamma)^2 + \zeta^2 < 1$. 

Furthermore, as noted in Remark \ref{rmk:independent of M}, these quantities can be made independent of the dimension $M$ of the active subspace, using  $0< C_{Q,M}\leq \mathrm{Trace}(Q)<\infty$. Then $\zeta^2 = \sigma^2 h\, \mathrm{Trace}(Q)$, and $h$ is chosen smaller such that $m = 2\lambda - \lambda^2 h - \sigma^2 \mathrm{Trace}(Q) > 0$.

     \begin{theorem}\label{thm:discrete_error}
Suppose Assumption \ref{assump:error_framework}-(1) holds, and the system parameters satisfy $(1-\gamma)^2 + \zeta^2 < 1$. For any $0 < \varepsilon < 1$, if the inverse temperature $\alpha > 0$ and the initial swarm variance parameter satisfy
\begin{equation}\label{eq:initial_condition_B}
    (1-\varepsilon)\mathbb{E}[e^{-\alpha \mathscr{E}(X_{in})}] \ge \frac{ \left( 2\gamma C_{\mathscr{E}} + L_\mathscr{E}(\gamma^2 + \zeta^2) \right) \alpha e^{-\alpha \mathscr{E}_{\min}}}{2(1 - ((1-\gamma)^2 + \zeta^2))} N D_{\text{pair}},
\end{equation}
then the common consensus state $X_\infty = \lim_{n \to \infty} X_n^i$ satisfies the error estimate
\begin{equation*}
    0 \le \essinf\limits_{\omega \in \Omega} \mathscr{E}(X_\infty) - \mathscr{E}_{\min} \le - \frac{\log \varepsilon}{\alpha} - \frac{1}{\alpha} \log \left( \mathbb{E}[e^{-\alpha (\mathscr{E}(X_{in}) - \mathscr{E}_{\min})}] \right).
\end{equation*}
\end{theorem}

\begin{remark}
One could observe that the right-hand side in \eqref{eq:initial_condition_B} admits a finite limit when $h\to 0$. Indeed
\begin{equation*}
\begin{aligned}
    \lim\limits_{h\to 0}\, \frac{ \left( 2\gamma C_{\mathscr{E}} + L_\mathscr{E}(\gamma^2 + \zeta^2) \right) }{2(1 - ((1-\gamma)^2 + \zeta^2))} 
    & = \frac{2\lambda C_{\mathscr{E}} + L_{\mathscr{E}}\sigma^2 C_{Q,M} }{2(2\lambda - \sigma^2 C_{Q,M})}\qquad \text{ if }\, 2\lambda > \sigma^2 C_{Q,M}\\
    & \leq \frac{2\lambda C_{\mathscr{E}} + L_{\mathscr{E}}\sigma^2 \mathrm{Trace}(Q) }{2(2\lambda - \sigma^2 \mathrm{Trace}(Q))}\quad \text{if }\, 2\lambda > \sigma^2 \mathrm{Trace}(Q),
\end{aligned}
\end{equation*}
where we used the bound in Remark \ref{rmk:independent of M} to get the upper-bound independent of $M$. Hence, a threshold from below in \eqref{eq:initial_condition_B} can be made uniform in both the discretization parameter $h$ (when small) and the dimension of the active space $M$.  
\end{remark}

\begin{proof}
We track the expected value of the exponentiated energy $e^{-\alpha \mathscr{E}(X_n^i)}$. Using the convexity property $e^y - e^x \ge e^x(y-x)$, we bound the forward difference path-by-path
\begin{equation}\label{eq:energy_diff_1}
    \frac{1}{N} \sum_{i=1}^N e^{-\alpha \mathscr{E}(X_{n+1}^i)} - \frac{1}{N} \sum_{i=1}^N e^{-\alpha \mathscr{E}(X_n^i)} \ge -\frac{\alpha}{N} \sum_{i=1}^N e^{-\alpha \mathscr{E}(X_n^i)} \big( \mathscr{E}(X_{n+1}^i) - \mathscr{E}(X_n^i) \big).
\end{equation}

Using the quadratic upper-bound condition \eqref{eq:quadratic_condition}, we bound the energy difference from above, which provides a lower bound for our negative term
\begin{equation*}
    \mathscr{E}(X_{n+1}^i) - \mathscr{E}(X_n^i) \le \langle \nabla \mathscr{E}(X_n^i), X_{n+1}^i - X_n^i \rangle_H + \frac{L_\mathscr{E}}{2} \|X_{n+1}^i - X_n^i\|_H^2.
\end{equation*}
Substituting this into \eqref{eq:energy_diff_1} yields
\begin{equation}\label{eq:energy_diff_2}
\begin{aligned}
    &\frac{1}{N} \sum_{i=1}^N e^{-\alpha \mathscr{E}(X_{n+1}^i)} - \frac{1}{N} \sum_{i=1}^N e^{-\alpha \mathscr{E}(X_n^i)} \\
    &\ge -\frac{\alpha}{N} \sum_{i=1}^N e^{-\alpha \mathscr{E}(X_n^i)} \left[ \langle \nabla \mathscr{E}(X_n^i), X_{n+1}^i - X_n^i \rangle_H + \frac{L_\mathscr{E}}{2} \|X_{n+1}^i - X_n^i\|_H^2 \right].
\end{aligned}
\end{equation}
We take the conditional expectation $\mathbb{E}[\,\cdot \mid \mathcal{F}_n]$ of both sides. Substituting the discrete update rule $X_{n+1}^i - X_n^i = -\gamma(X_n^i - \mathfrak{m}_n^{\alpha, N}) + \mathcal{N}_n^i$, the mean-zero property of the noise $\mathbb{E}[\mathcal{N}_n^i \mid \mathcal{F}_n] = 0$ (item (1) in Remark \ref{rmk: noise}) isolates the deterministic drift in the inner product. Additionally, using the conditional variance of the noise (item (2) in Remark \ref{rmk: noise}), the squared norm becomes  $(\gamma^2 + \zeta^2) \|X_n^i - \mathfrak{m}_n^{\alpha, N}\|_H^2$. Thus
\begin{equation}\label{eq:conditional_energy_step}
\begin{aligned}
    &\frac{1}{N} \sum_{i=1}^N \mathbb{E}[e^{-\alpha \mathscr{E}(X_{n+1}^i)} \mid \mathcal{F}_n] - \frac{1}{N} \sum_{i=1}^N e^{-\alpha \mathscr{E}(X_n^i)} \\
    &\ge \frac{\alpha}{N} \sum_{i=1}^N e^{-\alpha \mathscr{E}(X_n^i)} \left[ \gamma \langle \nabla \mathscr{E}(X_n^i), X_n^i - \mathfrak{m}_n^{\alpha, N} \rangle_H - \frac{L_\mathscr{E}}{2} (\gamma^2 + \zeta^2) \|X_n^i - \mathfrak{m}_n^{\alpha, N}\|_H^2 \right].
\end{aligned}
\end{equation}

To bound the inner product term, 
observe that by 
definition of the empirical consensus point $\sum e^{-\alpha \mathscr{E}(X_n^i)}(X_n^i - \mathfrak{m}_n^{\alpha, N}) = 0$, and we have
\begin{equation}\label{eq:gradient_vanish}
    \sum_{i=1}^N e^{-\alpha \mathscr{E}(X_n^i)} \langle \nabla \mathscr{E}(\mathfrak{m}_n^{\alpha, N}), X_n^i - \mathfrak{m}_n^{\alpha, N} \rangle_H = 0.
\end{equation}
Thus we can subtract \eqref{eq:gradient_vanish} from the right-hand side of \eqref{eq:conditional_energy_step} and put the gradient terms together. Then we apply the One-Sided Lipschitz condition \eqref{eq:osl_condition} to lower-bound the result
\begin{align*}
    \gamma \langle \nabla \mathscr{E}(X_n^i) - \nabla \mathscr{E}(\mathfrak{m}_n^{\alpha, N}), X_n^i - \mathfrak{m}_n^{\alpha, N} \rangle_H \ge -\gamma C_{\mathscr{E}} \|X_n^i - \mathfrak{m}_n^{\alpha, N}\|_H^2.
\end{align*}
Substituting this bound, and factoring out the squared distance while applying the uniform lower bound $\mathscr{E}(X_n^i) \ge \mathscr{E}_{\min} \implies e^{-\alpha \mathscr{E}(X_n^i)} \le e^{-\alpha \mathscr{E}_{\min}}$, we obtain
\begin{align*}
    &\frac{1}{N} \sum_{i=1}^N \mathbb{E}[e^{-\alpha \mathscr{E}(X_{n+1}^i)} \mid \mathcal{F}_n] - \frac{1}{N} \sum_{i=1}^N e^{-\alpha \mathscr{E}(X_n^i)} \\
    &\qquad \qquad \qquad \ge -\alpha \left( \gamma C_{\mathscr{E}} + \frac{L_\mathscr{E}}{2}(\gamma^2 + \zeta^2) \right) e^{-\alpha \mathscr{E}_{\min}} \frac{1}{N} \sum_{i=1}^N \|X_n^i - \mathfrak{m}_n^{\alpha, N}\|_H^2.
\end{align*}
Taking the unconditional expectation of both sides, the sum on the right yields the expected swarm variance $V_n$. Summing this recursive inequality telescopically from  $n=0$ to $\infty$, one gets
\begin{equation*}
    \mathbb{E} e^{-\alpha \mathscr{E}(X_\infty)} \ge \frac{1}{N} \sum_{i=1}^N \mathbb{E} e^{-\alpha \mathscr{E}(X_0^i)} - \alpha \left( \gamma C_{\mathscr{E}} + \frac{L_\mathscr{E}}{2}(\gamma^2 + \zeta^2) \right) e^{-\alpha \mathscr{E}_{\min}} \sum_{n=0}^\infty V_n.
\end{equation*}
By Corollary \ref{cor:variance_decay}, the variance decays exponentially as $V_n \le ((1-\gamma)^2 + \zeta^2)^n ND_{\text{pair}}$. Thus, the infinite geometric series converges to $\frac{ND_{\text{pair}}}{1 - ((1-\gamma)^2 + \zeta^2)}$. Substituting this finite sum and applying the well-preparedness condition \eqref{eq:initial_condition_B}, we guarantee
\begin{equation*}
    \mathbb{E} e^{-\alpha \mathscr{E}(X_\infty)} \ge \varepsilon\, \mathbb{E} e^{-\alpha \mathscr{E}(X_{in})}.
\end{equation*}
Because $e^{-\alpha \essinf\limits \mathscr{E}(X_\infty)} \ge \mathbb{E} e^{-\alpha \mathscr{E}(X_\infty)}$, we normalize the objective landscape by multiplying both sides by $e^{\alpha \mathscr{E}_{\min}}$, hence
\begin{equation*}
    e^{-\alpha (\essinf\limits \mathscr{E}(X_\infty) - \mathscr{E}_{\min})} \ge \varepsilon\, \mathbb{E} \left[ e^{-\alpha (\mathscr{E}(X_{in}) - \mathscr{E}_{\min})} \right].
\end{equation*}
Taking the natural logarithm of both sides and dividing by $-\alpha$ yields the final error estimate
\begin{equation*}
    \essinf\limits_{\omega \in \Omega} \mathscr{E}(X_\infty) - \mathscr{E}_{\min} \le - \frac{\log \varepsilon}{\alpha}- \frac{1}{\alpha} \log \left( \mathbb{E}[e^{-\alpha (\mathscr{E}(X_{in}) - \mathscr{E}_{\min})}] \right)\,.
\end{equation*}
\end{proof}

\begin{lemma}[Laplace Principle]
\label{lem:asymptotic_laplace}
Suppose the initial random variable $X_{in} \in V$ is well-prepared, meaning it assigns strictly positive probability mass to any arbitrary energy neighborhood of the global infimum (see Assumption \ref{assump:error_framework}-(2)). Specifically, for any $\delta > 0$, define the sublevel set $B_\delta := \{x \in V : \mathscr{E}(x) \le \mathscr{E}_{\min} + \delta\}$, and assume that
\begin{equation}\label{eq:pos_mass_assumption}
    \mathbb{P}(X_{in} \in B_\delta) > 0 \quad \text{for all } \delta > 0.
\end{equation}
Then, the logarithmic expectation converges to the global infimum as the inverse temperature $\alpha \to \infty$
\begin{equation*}
    \lim_{\alpha \to \infty} -\frac{1}{\alpha} \log \left( \mathbb{E} \left[ e^{-\alpha \mathscr{E}(X_{in})} \right] \right) = \mathscr{E}_{\min}.
\end{equation*}
\end{lemma}

\begin{proof}
We establish the limit by proving matching lower and upper bounds via a squeeze argument.

\textbf{Step 1.} \textit{(The Lower Bound.)} \\
By the definition of the global infimum, $\mathscr{E}(X_{in}) \ge \mathscr{E}_{\min}$ almost surely. Because the exponential function $y \mapsto e^{-\alpha y}$ is strictly decreasing for $\alpha > 0$, we have the pathwise upper bound
\begin{equation*}
    e^{-\alpha \mathscr{E}(X_{in})} \le e^{-\alpha \mathscr{E}_{\min}} \quad \text{a.s.}
\end{equation*}
Taking the expectation of both sides preserves the inequality
\begin{equation*}
    \mathbb{E} \left[ e^{-\alpha \mathscr{E}(X_{in})} \right] \le e^{-\alpha \mathscr{E}_{\min}}.
\end{equation*}
Applying the natural logarithm to both sides (which is monotonically increasing) and dividing by $-\alpha$ (which reverses the inequality direction) yields
\begin{equation}\label{eq:laplace_lower_bound}
    -\frac{1}{\alpha} \log \left( \mathbb{E} \left[ e^{-\alpha \mathscr{E}(X_{in})} \right] \right) \ge \mathscr{E}_{\min}.
\end{equation}
This holds for all $\alpha > 0$, thus we have  $\liminf_{\alpha \to \infty} -\frac{1}{\alpha} \log ( \mathbb{E} [ e^{-\alpha \mathscr{E}(X_{in})} ] ) \ge \mathscr{E}_{\min}$.

\textbf{Step 2.} \textit{(The Upper Bound.)} \\
Fix an arbitrarily small $\delta > 0$. Because the exponential function is strictly positive, we establish a strict lower bound on the expectation by restricting the integration domain to the sublevel set $B_\delta$
\begin{equation*}
    \mathbb{E} \left[ e^{-\alpha \mathscr{E}(X_{in})} \right] \ge \mathbb{E} \left[ e^{-\alpha \mathscr{E}(X_{in})} \mathbf{1}_{B_\delta}(X_{in}) \right].
\end{equation*}
For all realizations where $X_{in} \in B_\delta$, the energy satisfies $\mathscr{E}(X_{in}) \le \mathscr{E}_{\min} + \delta$. Consequently, the exponential weight is bounded below by $e^{-\alpha (\mathscr{E}_{\min} + \delta)}$. Factoring this deterministic minimum out of the restricted expectation gives
\begin{equation*}
    \mathbb{E} \left[ e^{-\alpha \mathscr{E}(X_{in})} \mathbf{1}_{B_\delta}(X_{in}) \right] \ge e^{-\alpha (\mathscr{E}_{\min} + \delta)} \mathbb{E}[\mathbf{1}_{B_\delta}(X_{in})] = e^{-\alpha (\mathscr{E}_{\min} + \delta)} \mathbb{P}(X_{in} \in B_\delta).
\end{equation*}
Taking the natural logarithm of both sides yields
\begin{equation*}
    \log \left( \mathbb{E} \left[ e^{-\alpha \mathscr{E}(X_{in})} \right] \right) \ge -\alpha(\mathscr{E}_{\min} + \delta) + \log \big( \mathbb{P}(X_{in} \in B_\delta) \big).
\end{equation*}
Dividing by $-\alpha$ isolates the desired functional form and reverses the inequality
\begin{equation}\label{eq:laplace_upper_bound}
    -\frac{1}{\alpha} \log \left( \mathbb{E} \left[ e^{-\alpha \mathscr{E}(X_{in})} \right] \right) \le \mathscr{E}_{\min} + \delta - \frac{1}{\alpha} \log \big( \mathbb{P}(X_{in} \in B_\delta) \big).
\end{equation}

\textbf{Step 3.} \textit{(The Asymptotic Limit.)} \\
By Assumption \eqref{eq:pos_mass_assumption}, the probability mass $\mathbb{P}(X_{in} \in B_\delta)$ is a strictly positive constant independent of $\alpha$. Therefore, its logarithm is finite. Taking the limit supremum of \eqref{eq:laplace_upper_bound} as $\alpha \to \infty$ eliminates the probability penalty term
\begin{equation*}
    \limsup_{\alpha \to \infty} -\frac{1}{\alpha} \log \left( \mathbb{E} \left[ e^{-\alpha \mathscr{E}(X_{in})} \right] \right) \le \mathscr{E}_{\min} + \delta - \lim_{\alpha \to \infty} \frac{\log \big( \mathbb{P}(X_{in} \in B_\delta) \big)}{\alpha} = \mathscr{E}_{\min} + \delta.
\end{equation*}
Because $\delta > 0$ was chosen arbitrarily, we may take the limit as $\delta \to 0$ to conclude
\begin{equation*}
    \limsup_{\alpha \to \infty} -\frac{1}{\alpha} \log \left( \mathbb{E} \left[ e^{-\alpha \mathscr{E}(X_{in})} \right] \right) \le \mathscr{E}_{\min}.
\end{equation*}
Combining this upper bound with the absolute lower bound \eqref{eq:laplace_lower_bound} completes the proof.
\end{proof}

\begin{corollary}[Convergence to Global Minimum] \label{cor:asymptotic_convergence}
Suppose the conditions of Theorem \ref{thm:discrete_error} and Lemma \ref{lem:asymptotic_laplace} hold. For any arbitrarily small target error $\eta > 0$ and any fixed $\varepsilon \in (0, 1)$, there exists a sufficiently large inverse temperature $\alpha > 0$ such that 
\begin{equation} \label{eq:upper_bound_eta}
    - \frac{\log \varepsilon}{\alpha} - \frac{1}{\alpha} \log \left( \mathbb{E} \left[ e^{-\alpha (\mathscr{E}(X_{in}) - \mathscr{E}_{\min})} \right] \right) < \eta.
\end{equation}
Consequently, if the initial swarm is drawn such that its maximum expected pairwise variance $D_{\text{pair}}$ (as defined in Corollary \ref{cor:variance_decay}) is sufficiently small to satisfy the well-preparedness condition \eqref{eq:initial_condition_B} evaluated at this specific $\alpha$, the common consensus state $X_\infty$ guarantees arbitrary precision to the global minimum
\begin{equation}
    0 \le  \essinf\limits_{\omega \in \Omega} \mathscr{E}(X_\infty) - \mathscr{E}_{\min} < \eta.
\end{equation}
\end{corollary}

\begin{proof}
We evaluate the asymptotic behavior of the upper bound from Theorem \ref{thm:discrete_error}. The logarithmic expectation can be decomposed as
\begin{equation*}
    - \frac{1}{\alpha} \log \left( \mathbb{E} \left[ e^{-\alpha (\mathscr{E}(X_{in}) - \mathscr{E}_{\min})} \right] \right) = - \frac{1}{\alpha} \log \left( \mathbb{E} \left[ e^{-\alpha \mathscr{E}(X_{in})} \right] \right) - \mathscr{E}_{\min}.
\end{equation*}
By the asymptotic Laplace Principle established in Lemma \ref{lem:asymptotic_laplace}, the limit of the first term on the right-hand side evaluates to the global infimum $\mathscr{E}_{\min}$. Therefore, taking the limit of the entire expression as $\alpha \to \infty$ yields:
\begin{equation*}
    \lim_{\alpha \to \infty} \left( - \frac{1}{\alpha} \log \left( \mathbb{E} \left[ e^{-\alpha (\mathscr{E}(X_{in}) - \mathscr{E}_{\min})} \right] \right) \right) = \mathscr{E}_{\min} - \mathscr{E}_{\min} = 0.
\end{equation*}

Furthermore, for any fixed $\varepsilon \in (0, 1)$ it holds $\lim_{\alpha \to \infty} \frac{-\log \varepsilon}{\alpha} = 0$. Because the sum of the limits of these two terms is  zero, there exists a finite threshold $\alpha^* > 0$ such that for all $\alpha > \alpha^*$, this sum falls below the target error $\eta$.

Fixing the temperature parameter at such an $\alpha > \alpha^*$ reduces the right-hand side of the well-preparedness condition \eqref{eq:initial_condition_B} to a finite, deterministic constant proportional to $N D_{\text{pair}}$. By scaling the initial swarm distribution such that the pairwise spatial variance $D_{\text{pair}}$ is sufficiently small, the condition is satisfied, validating the error estimate of Theorem \ref{thm:discrete_error} and concluding the proof.
\end{proof}

\section{Numerical Illustrations}\label{sec:numerics}

\subsection{Example 1: Optimal Control of a PDE }\label{sec:control of pde}

To validate the algorithmic convergence bounds established in Theorem \ref{thm:discrete_error} without the interference of spatial truncation errors, we construct an example where the exact continuous minimizer $u^*$ resides entirely within the active subspace $V$. 

Let $\Omega = (0,1) \subset \mathbb{R}$ be a bounded domain. 
We consider the following Inverse Source Problem. 
This is an optimal control problem of a 1D elliptic PDE (Poisson equation), where the objective is to reconstruct a source term $u$ that drives the system state $y$ toward a target state $y_{\text{d}}$:
\begin{equation}\label{ex:control pde}
\begin{aligned}
    & \min\limits_{u \in L^2(\Omega)} \mathscr{E}(u) := \frac{1}{2}\|y - y_{\text{d}}\|_{L^2(\Omega)}^2 + \frac{\nu}{2}\|u\|_{L^2(\Omega)}^2, \qquad  \text{ for some } \nu > 0,\\
    & \text{subject to: }\; -\Delta y = u \quad \text{in } \Omega,\quad \text{ and } \; y = 0 \quad \text{on } \partial\Omega.
\end{aligned}
\end{equation}
We define the infinite-dimensional Hilbert space $H := L^2(\Omega)$, equipped with the standard inner product
\[
\langle u,v \rangle_H := \int_\Omega u(x)v(x)\,\dd x.
\]
The active subspace $V \subset H$ is spanned by the first $M$ orthonormal Fourier sine modes, which correspond to the eigenfunctions of the Dirichlet Laplacian: $\phi_k(x) = \sqrt{2} \sin(k \pi x)$ for $k = 1, \dots, M$. The corresponding eigenvalues are $\mu_k = (k\pi)^2$. 

By substituting the explicit spectral solution of the state equation $y(x) = \sum \frac{u_k}{\mu_k} \phi_k(x)$ into the cost functional, the optimality condition in the frequency domain ensures that the exact global minimizer $u^*$ has Fourier coefficients:
\begin{equation}
    u_k^* = \frac{d_k \mu_k}{1 + \nu \mu_k^2}, \quad \text{where } d_k = \langle y_{\text{d}}, \phi_k \rangle_H.
\end{equation}
To systematically eliminate spatial truncation errors and significantly reduce computational complexity, the objective functional is evaluated entirely in the frequency domain. Because the basis functions $\{\phi_k\}$ are the exact orthonormal eigenfunctions of the Dirichlet Laplacian with corresponding eigenvalues $\mu_k = (k\pi)^2$, the state equation $-\Delta y = u$ can be inverted analytically within the active subspace, yielding the spectral state coefficients $y_k = u_k/\mu_k$. By invoking Parseval's identity, the continuous $L^2(\Omega)$ norms defining the cost functional are mapped isometrically to discrete $\ell^2$ norms of the Fourier coefficients. Consequently, the energy of any particle state $X_n^i$ can be computed exactly as
$$ \mathscr{E}(X_n^i) = \frac{1}{2} \sum_{k=1}^M \left( \left| \frac{X_{n,k}^i}{\mu_k} - d_k \right|^2 + \nu |X_{n,k}^i|^2 \right), $$
completely bypassing the need for finite-difference grid discretizations and numerical integration.

To guarantee the infinite-dimensional minimizer satisfies $u^* \in V$, we construct the target state $y_{\text{d}}(x)$ such that $d_k = 0$ for all $k > M$. In our numerical experiments, we set $M=10$ and define the target state  using the first three modes: 
\[
y_{\text{d}}(x) = 2\phi_1(x) - \phi_2(x) + 0.5\phi_3(x).
\]
Because the target state contains no frequencies higher than $M$, the high-frequency components of the exact optimal control are identically zero ($u_k^* = 0$ for $k > M$). Consequently, the active subspace projection introduces zero Galerkin truncation error, allowing us to observe the exponential energy decay of the discrete-time CBO scheme.

\begin{algorithm}[h]
\caption{Spectral Discrete-Time CBO for Inverse Source PDE Problem}
\label{alg:cbo_pde}
\begin{algorithmic}[1]
\renewcommand{\algorithmicrequire}{\textbf{Input:}}
\renewcommand{\algorithmicensure}{\textbf{Output:}}
\REQUIRE Swarm size $N$, max iterations $N_{\text{iter}}$, CBO step size $h = \Delta t$.
\REQUIRE Target state Fourier coefficients $d_k = \langle y_{\text{d}}, \phi_k \rangle_H$, penalty $\nu$.
\REQUIRE CBO parameters: $\lambda$ (drift), $\sigma$ (noise), $\alpha$ (inverse temp), $M$ (Fourier modes).
\STATE \textbf{Initialize:} 
\STATE Define Dirichlet Laplacian eigenvalues $\mu_k = (k\pi)^2$ and trace-class noise eigenvalues $\beta_k = 1/k^2$.
\STATE Draw initial particles $\{X_0^i\}_{i=1}^N$ within the active subspace $V$, parameterized by their Fourier coefficients $X_{0,k}^i$.
\FOR{$n = 0$ \TO $N_{\text{iter}} - 1$}
    \STATE \textcolor{blue}{\textit{\% 1. Evaluate Objective Functional via Parseval's Identity}}
    \FOR{$i = 1$ \TO $N$}
        \STATE Evaluate exact energy: $\mathscr{E}(X_n^i) = \frac{1}{2} \sum_{k=1}^M \left( \left| \frac{X_{n,k}^i}{\mu_k} - d_k \right|^2 + \nu |X_{n,k}^i|^2 \right)$.
    \ENDFOR
    
    \STATE \textcolor{blue}{\textit{\% 2. Compute Stable Gibbs Weights (Log-Sum-Exp)}}
    \STATE Find $M_{\max} = \max_{j} (-\alpha \mathscr{E}(X_n^j))$.
    \FOR{$i = 1$ \TO $N$}
        \STATE $\omega_n^i = \frac{\exp(-\alpha \mathscr{E}(X_n^i) - M_{\max})}{\sum_{\ell=1}^N \exp(-\alpha \mathscr{E}(X_n^\ell) - M_{\max})}$.
    \ENDFOR
    \STATE Update consensus point: $\mathfrak{m}_n^{\alpha, N} = \sum_{i=1}^N \omega_n^i X_n^i$ (computed coordinate-wise).
    
    \STATE \textcolor{blue}{\textit{\% 3. Generate Common Active Subspace Noise (in Coefficient Space)}}
    \FOR{$k = 1$ \TO $M$}
        \STATE Draw independent standard normal variable: $\zeta_{n,k} \sim \mathcal{N}(0,1)$.
        \STATE Compute the noise coefficient: $W_{n,k} = \sqrt{\beta_k}\, \zeta_{n,k}$.
    \ENDFOR
    
    \STATE \textcolor{blue}{\textit{\% 4. Update Particle Swarm (in Coefficient Space)}}
    \FOR{$i = 1$ \TO $N$}
        \STATE Compute exact $L^2$-distance: $D_n^i = \sqrt{ \sum_{\ell=1}^M (X_{n,\ell}^i - \mathfrak{m}_{n,\ell}^{\alpha, N})^2 }$.
        \FOR{$k = 1$ \TO $M$}
            \STATE Scale common noise: $\mathcal{N}_{n,k}^i = \sigma \sqrt{h} D_n^i W_{n,k}$.
            \STATE Euler step: $X_{n+1,k}^i = X_{n,k}^i - \lambda h (X_{n,k}^i - \mathfrak{m}_{n,k}^{\alpha, N}) + \mathcal{N}_{n,k}^i$.
        \ENDFOR
    \ENDFOR
\ENDFOR
\vspace{0.1cm}
\ENSURE The final consensus control $\mathfrak{m}_{N_{\text{iter}}}^{\alpha, N}$.
\end{algorithmic}
\end{algorithm}

In our numerical experiments, the physical domain $\Omega = (0,1)$ is discretized uniformly using $N_x = 500$ grid points, and the control penalty parameter is set to $\nu = 0.01$. The active subspace $V$ is truncated to the first $M = 10$ Fourier sine modes, with trace-class noise eigenvalues prescribed as $\beta_k = 1/k^2$ to ensure the well-posedness of the diffusion process in the Hilbert space. The CBO algorithm is executed with a swarm size of $N = 500$ particles over $N_{\text{iter}} = 500$ iterations, utilizing a time step of $\Delta t = 0.01$ to guarantee the stability of the Euler-Maruyama discretization. To circumvent the convex hull trap (ensuring the true global minimizer lies within the swarm's initial search radius) the particles are initialized in the coefficient space using a broad Gaussian distribution, $\mathcal{N}(0, 15^2)$. Finally, the consensus dynamics are governed by a unit drift parameter $\lambda = 1.0$, a reduced noise scaling $\sigma = 0.2$, and an inverse temperature $\alpha = 10^9$. This specific parameter configuration is explicitly chosen to emphasize the exploitation phase; it enforces the application of Laplace principle, driving a rapid concentration of the Gibbs measure and ensuring the strict log-linear variance decay observed in the swarm's consensus trajectory.


\begin{figure}[htbp]
    \centering
    \makebox[\textwidth][c]{%
        \begin{subfigure}[b]{0.36\textwidth}
            \centering
            \includegraphics[width=\textwidth]{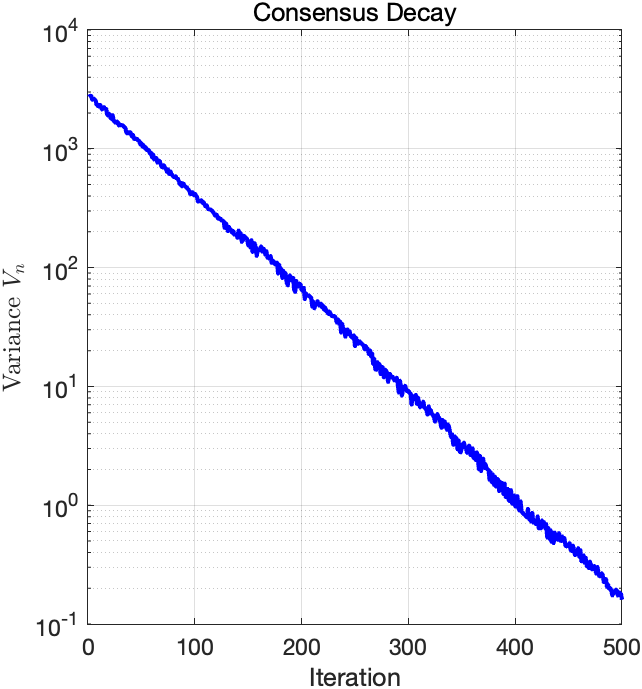}
            \label{fig:left}
        \end{subfigure}
        \hspace{0.015\textwidth}
        \begin{subfigure}[b]{0.36\textwidth}
            \centering
            \includegraphics[width=\textwidth]{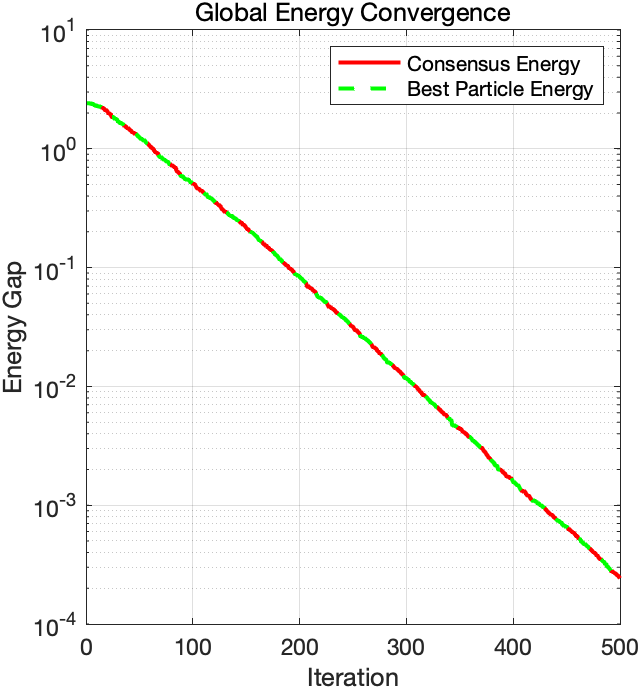}
            \label{fig:middle}
        \end{subfigure}
        \hspace{0.015\textwidth}
        \begin{subfigure}[b]{0.36\textwidth}
            \centering
            \includegraphics[width=\textwidth]{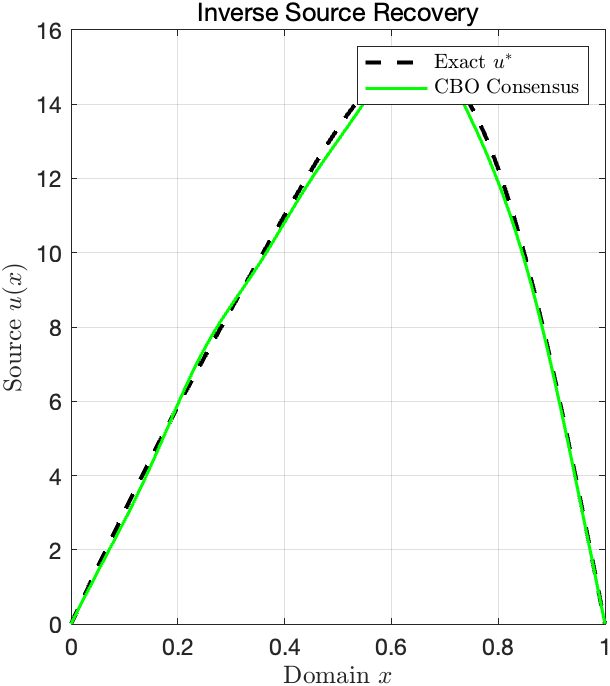}
            \label{fig:right}
        \end{subfigure}%
    }
    \caption{Numerical validation of the coefficient-space CBO algorithm for the inverse source problem. \textbf{Left:} Exponential decay of the empirical swarm variance. \textbf{Middle:} Convergence of the consensus and best-particle energy gaps relative to the exact continuous minimum, which coincides with the Galerkin subspace minimum in this example. \textbf{Right:} Comparison of the final reconstructed CBO consensus profile with the analytical minimizer $u^*$.}
    \label{fig:inverse_source_results}
\end{figure}


Figure \ref{fig:inverse_source_results} presents the numerical validation of the coefficient-space CBO algorithm applied to the inverse source problem. The left panel illustrates the evolution of the empirical swarm variance
\begin{equation*}
    V_n^{\mathrm{emp}} := \frac{1}{N}\sum_{i=1}^N \left\|X_n^i-\mathfrak{m}_n^{\alpha,N}\right\|_H^2,
\end{equation*}
which exhibits an approximately log-linear decay over $500$ iterations, in agreement with the exponential consensus behavior established by the theoretical analysis. The middle panel shows the convergence of the objective values. Specifically, the plotted ``Consensus Energy'' and ``Best Particle Energy'' represent the residual energy gaps $\mathscr{E}(\mathfrak{m}_n^{\alpha,N})-\mathscr{E}(u^*)$ and $\min_{1\leq i\leq N}\mathscr{E}(X_n^i)-\mathscr{E}(u^*)$, respectively. Since the prescribed target state has nonzero Fourier coefficients only in the first three modes, the analytical minimizer $u^*$ belongs to the active subspace for every $M\geq 3$. Consequently, the Galerkin subspace minimum coincides with the exact minimum of the original infinite-dimensional problem, and the plotted quantities therefore measure the energy gaps relative to the exact continuous minimum. The consensus and best-particle energy curves closely track each other and decay to a residual energy gap of approximately $10^{-4}$, indicating that the Gibbs-weighted consensus successfully identifies a neighborhood of the global minimizer while retaining stochastic exploration during the optimization process. Finally, the right panel compares the spatial profile of the reconstructed source term with the analytical minimizer $u^*$. The final CBO consensus solution shows close agreement with $u^*$, further demonstrating the ability of the proposed Hilbert-space CBO method to approximate the solution of the underlying infinite-dimensional PDE-constrained optimization problem.

\subsection{Example 2: Dimension Sensitivity}

To demonstrate the necessity of our infinite-dimensional formulation and its capacity to overcome the curse of dimensionality, we introduce a secondary numerical experiment designed to test the algorithm's stability as the active subspace dimension $M$ increases.  For this numerical experiment, we use the same example as in the previous subsection \S \ref{sec:control of pde}, for which the algorithmic parameters are configured as follows: $N = 20{,}000$, $N_{\text{iter}} = 500$, $\Delta t = 0.01$, $\nu = 0.01$, $\lambda = 1$, $\sigma = 0.1$, and $\alpha = 10^9$.

While the exact solution in the previous example \eqref{ex:control pde} was represented by just three Fourier modes, we now consider a target state with infinite non-zero Fourier coefficients. Specifically, we define a discontinuous step function as the target state:
\[
y_{\text{d}}(x) = 1 \quad \text{for}\quad x \in [0.25, 0.75], \qquad \text{ and }\quad  y_{\text{d}}(x) = 0 \quad \text{otherwise.}
\]
Because the corresponding Fourier coefficients $d_k$ decay slowly, the exact reconstruction of the sharp corners of the source requires a large truncation dimension in order to minimize the Galerkin truncation error.

    

\begin{figure}[htbp]
    \centering
    \makebox[\textwidth][c]{%
        \begin{subfigure}[b]{0.55\textwidth}
            \centering
            \includegraphics[width=\textwidth]{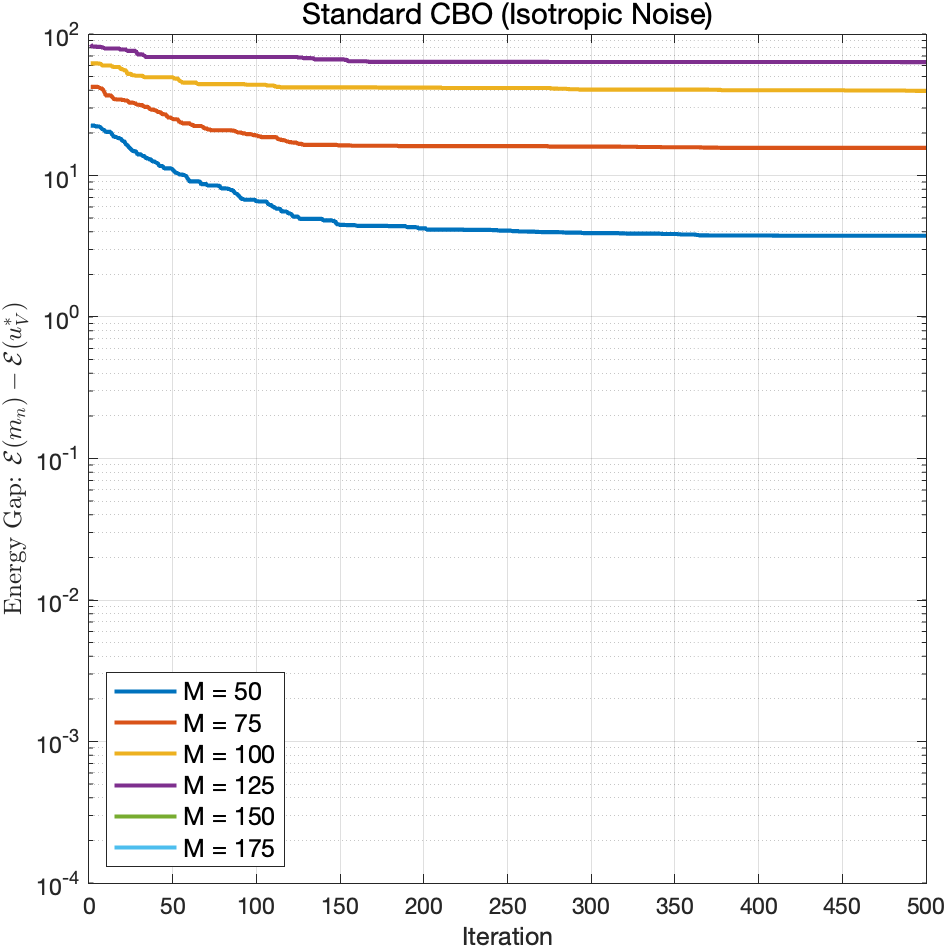}
            \label{fig:scaling_left}
        \end{subfigure}
        \hspace{0.02\textwidth}
        \begin{subfigure}[b]{0.55\textwidth}
            \centering
            \includegraphics[width=\textwidth]{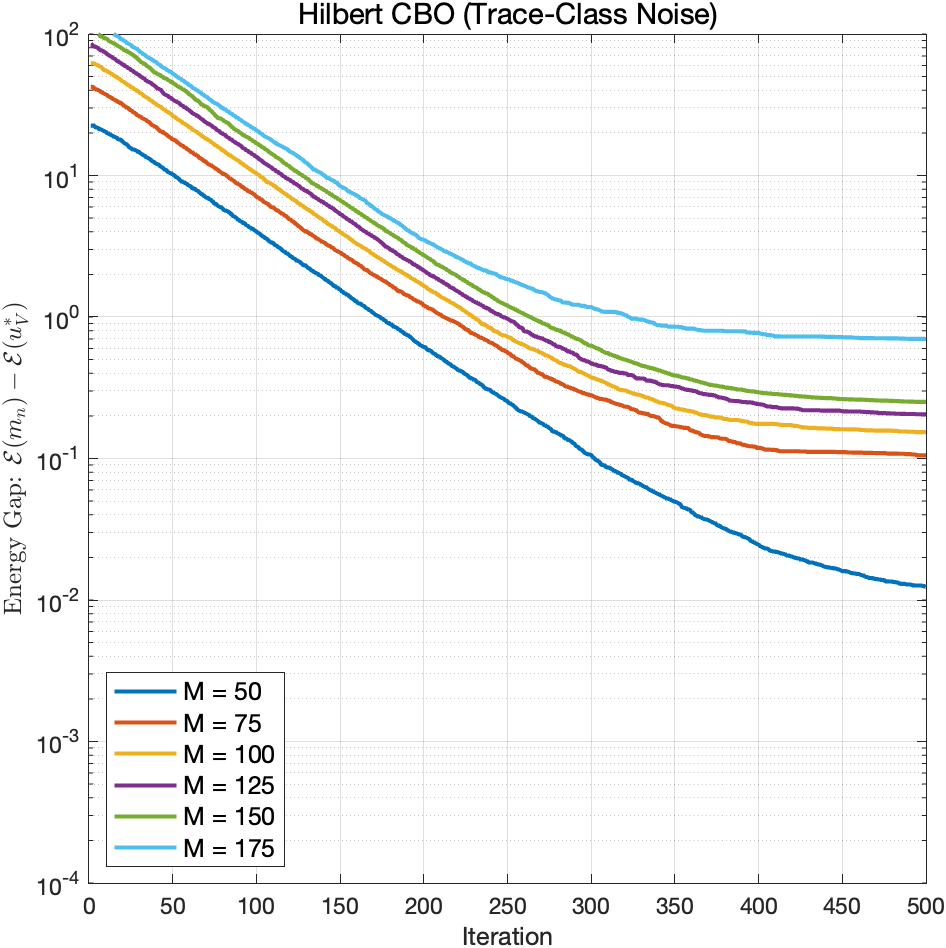}
            \label{fig:scaling_right}
        \end{subfigure}%
    }
    \caption{Comparative dimensional-scaling test for the inverse source problem \eqref{ex:control pde} with a discontinuous step-function target state. \textbf{Left:} \textit{Standard CBO with isotropic noise}. The algorithm fails to achieve energy descent for the higher-dimensional approximations ($M\geq 50$), as the total stochastic variance increases with the ambient dimension and leads to excessive dispersion of the particle swarm. \textbf{Right:} \textit{Hilbert-space CBO with trace-class noise ($\beta_k=1/k^2$)}. The total stochastic variance remains uniformly bounded with respect to $M$, preserving the optimization dynamics and enabling consistent energy descent as the Galerkin dimension increases, including at $M=175$.}
    \label{fig:dimension_scaling}
\end{figure}

Figure \ref{fig:dimension_scaling} presents a side-by-side comparison between standard finite-dimensional CBO with isotropic exploration noise and the proposed Hilbert-space CBO with trace-class noise, evaluated over a range of Galerkin dimensions from $M=50$ to $M=175$. In the standard isotropic formulation, all directions are excited with equal intensity. Consequently, the total stochastic variance injected into the particle system scales linearly with the ambient dimension, i.e., as $\mathcal{O}(M)$. As illustrated in the left panel of Figure \ref{fig:dimension_scaling}, this increasing stochastic contribution substantially deteriorates the optimization dynamics as $M$ grows. The resulting dispersion of the particle swarm prevents a systematic decrease of the objective value, and the dynamics stabilize at suboptimal energy levels without exhibiting meaningful energy descent.

In contrast, the proposed Hilbert-space CBO employs a trace-class covariance structure with eigenvalues $\beta_k=1/k^2$. Consequently, the total stochastic variance satisfies
\begin{equation*}
    C_{Q,M}=\sum_{k=1}^M\beta_k\leq\operatorname{Trace}(Q)<\infty,
\end{equation*}
uniformly with respect to the Galerkin dimension $M$; see Remark \ref{rmk:independent of M}. The decay of the covariance eigenvalues reduces stochastic excitation in the higher-frequency modes while retaining exploration along the dominant low-frequency directions. This behavior is reflected in the right panel of Figure \ref{fig:dimension_scaling}: the energy decay curves remain remarkably stable as $M$ increases from $50$ to $175$, with comparable convergence behavior across the tested spatial resolutions. The experiment therefore illustrates the dimension-robust character of the trace-class exploration mechanism and its relevance for function-space optimization, where increasing the Galerkin dimension is necessary to obtain increasingly accurate approximations of the underlying infinite-dimensional problem.

\begin{figure}[htbp]
    \centering
    \includegraphics[width=0.8\textwidth]{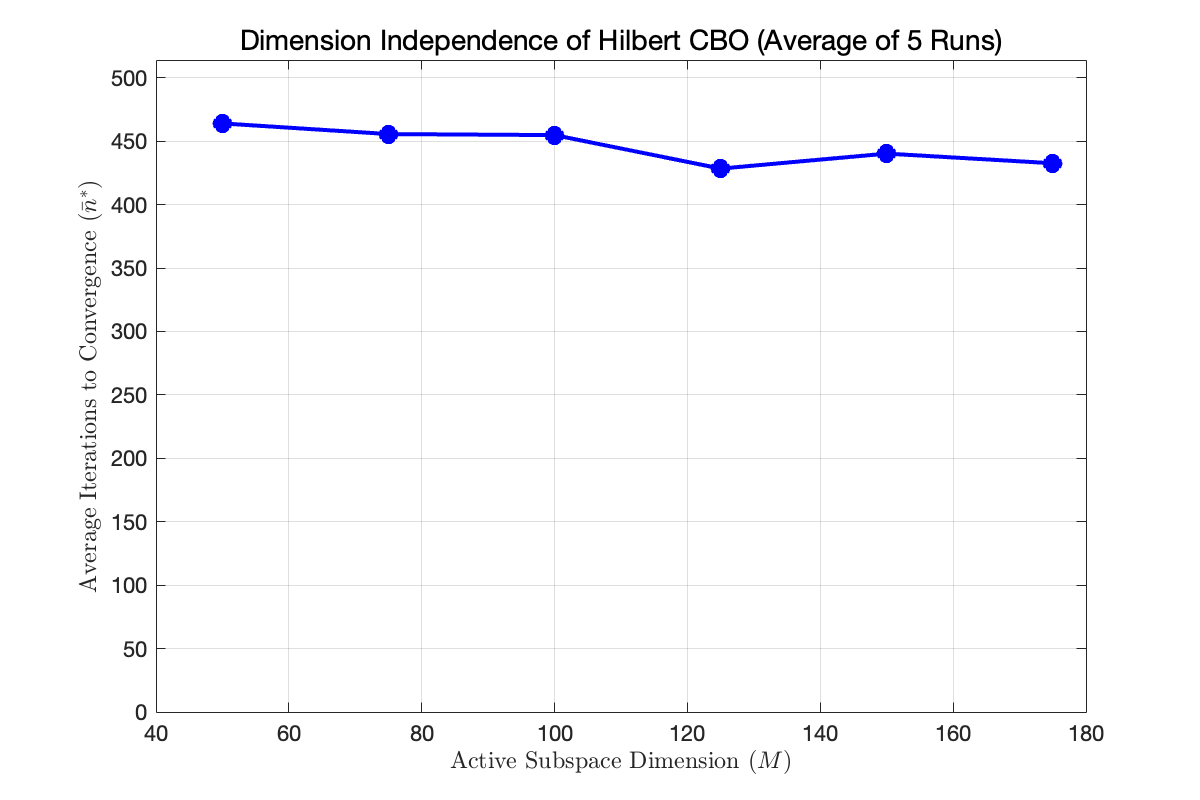}
    \caption{\textit{Dimension-robust convergence behavior of the proposed Hilbert-space CBO algorithm}. Average number of iterations $\bar{n}^*$ required to approach the terminal energy level as a function of the active subspace dimension $M$, computed over five realizations. The nearly constant iteration count across the tested dimensions indicates that the temporal convergence behavior is largely insensitive to the Galerkin dimension, consistently with the dimension-uniform control of the stochastic variance provided by the trace-class covariance structure.}
    \label{fig:iterations_convergence}
\end{figure}

To further investigate the dependence of the convergence behavior on the active subspace dimension, Figure \ref{fig:iterations_convergence} reports the number of iterations required by the Hilbert-space CBO algorithm to approach its terminal energy level for different values of $M$. For each realization, we denote by
\begin{equation*}
    E_{\mathrm{gap}}(n) := \mathscr{E}(\mathfrak{m}_n^{\alpha,N})-\mathscr{E}_{\min}
\end{equation*}
the consensus energy gap at iteration $n$, where $\mathfrak{m}_n^{\alpha,N}$ is the Gibbs-weighted consensus point and $\mathscr{E}_{\min}$ denotes the exact minimum of the objective over the active subspace. We estimate the terminal plateau value by averaging $E_{\mathrm{gap}}(n)$ over the final $50$ iterations,
\begin{equation*}
    E_{\mathrm{plat}} := \frac{1}{50} \sum_{n=N_{\mathrm{iter}}-49}^{N_{\mathrm{iter}}} E_{\mathrm{gap}}(n),
\end{equation*}
and define the convergence iteration as
\begin{equation*}
    n^* := \min\left\{n:\,
    E_{\mathrm{gap}}(n)-E_{\mathrm{plat}}
    \leq 0.05 \,E_{\mathrm{plat}}
    \right\}
    =
    \min\left\{n:\,
    E_{\mathrm{gap}}(n)
    \leq 1.05\,E_{\mathrm{plat}}
    \right\}.
\end{equation*}
Thus, $n^*$ is the first iteration at which the consensus energy gap is no more than $5\%$ above its estimated terminal plateau value. To account for the stochastic nature of the algorithm, this quantity is computed for five realizations at each dimension $M$, and the reported value $\bar{n}^*$ is their average.

As shown in Figure \ref{fig:iterations_convergence}, the average convergence iteration $\bar{n}^*$ remains approximately constant as the active subspace dimension increases from $M=50$ to $M=175$. This observation complements the dimension-uniform variance bound $C_{Q,M}\leq\operatorname{Trace}(Q)$ and indicates that, over the range of dimensions considered, increasing the spatial resolution does not lead to a noticeable increase in the number of iterations required to approach the terminal energy level. The experiment therefore provides numerical evidence of dimension-robust temporal convergence for the proposed trace-class CBO formulation.

\subsection{Example 3: Elliptic Energy with Mixed Boundary Conditions}

To further validate the proposed algorithm on problems with spatially varying coefficients and complex boundary constraints, we consider the direct minimization of an elliptic energy functional. Let the physical domain be the unit square $\Omega = (0,1)^2$. We impose mixed boundary conditions by partitioning the boundary $\partial\Omega$ into a Dirichlet boundary 
\[
\Gamma_D = \{(0,y) : 0\le y\le 1\} \cup \{(x,0) : 0\le x\le 1\}
\]
and a Neumann boundary 
\[
\Gamma_N = \{(1,y) : 0\le y\le 1\} \cup \{(x,1) : 0\le x\le 1\}.
\]

The search space is the admissible Sobolev subspace 
\[
H = \{ v\in H^1(\Omega) : v=0 \text{ on } \Gamma_D \}.
\]
The optimization problem seeks to minimize the objective functional
\begin{equation*}
    \min_{v\in H}\; J(v) := \frac{1}{2} \int_{\Omega} \left[ (1+x+y)|\nabla v|^2+v^2 \right]\,\dd x\dd y - \int_{\Omega} f\,v\,\dd x\,\dd y - \int_{\Gamma_N} g\,v\,\dd s
\end{equation*}
for some given data $f,g$. Equivalently, the global minimizer $u\in H$ is the weak solution satisfying the variational problem
\begin{equation*}
    \int_{\Omega} (1+x+y)\nabla u\cdot\nabla w\,\dd x\,\dd y + \int_{\Omega}u\,w\,\dd x\,\dd y = \int_{\Omega}f\,w\,\dd x\,\dd y + \int_{\Gamma_N}g\,w\,\dd s,
\end{equation*}
for all test functions $w\in H$.

\subsubsection*{Benchmark Construction}

For numerical validation, we construct a benchmark problem admitting a known exact solution and consider the corresponding strong form of the elliptic boundary value problem
\begin{equation*}
\left\{\;
\begin{aligned}
    -\nabla\cdot\left((1+x+y)\nabla u\right)+u &= f, && \quad \text{in }\Omega, \\
    u &= 0, && \quad \text{on }\Gamma_D, \\
    (1+x+y)\nabla u\cdot n &= g, && \quad \text{on }\Gamma_N.
\end{aligned}
\right.
\end{equation*}
We manufacture the exact analytical solution $u_{\mathrm{exact}}(x,y) = \sin(\pi x)\sin(\pi y)$. 
By substituting $u_{\mathrm{exact}}$ into the differential operator, we define the corresponding volume source term
\begin{equation*}
\begin{aligned}
f(x,y) &= 2\pi^2(1+x+y)\sin(\pi x)\sin(\pi y) \\
    &\quad - \pi\cos(\pi x)\sin(\pi y) - \pi\sin(\pi x)\cos(\pi y) + \sin(\pi x)\sin(\pi y).
\end{aligned}
\end{equation*}
Similarly, the Neumann boundary data is evaluated as
\begin{equation*}
g(x,y) =
\begin{cases}
-\pi(2+y)\sin(\pi y), & \text{on } x=1, \\
-\pi(2+x)\sin(\pi x), & \text{on } y=1.
\end{cases}
\end{equation*}
By construction, $u_{\mathrm{exact}} = \argmin_{v\in H}J(v)$. This formulation allows us to directly evaluate the CBO swarm's capacity to minimize an energy functional in $H^1(\Omega)$ where the spatially varying coefficient $(1+x+y)$ prevents perfect spectral diagonalization.

\subsubsection*{Numerical Configuration and Basis Selection}

To execute the numerical minimization for the mixed boundary value problem, the Hilbert CBO algorithm is configured with a swarm size of $N = 20{,}000$ particles evaluated over $N_{\text{iter}} = 800$ iterations. The consensus dynamics are driven by a time step of $\Delta t = 0.01$, an alignment drift rate of $\lambda = 1.0$, an exploration noise intensity of $\sigma = 0.1$, and an inverse temperature of $\alpha = 10^9$. As with the previous example, the noise covariance is scaled by trace-class eigenvalues $\beta_{m} = 1/m^2$ to ensure well-posedness in the infinite-dimensional setting.

The optimization is performed within a Galerkin active subspace of dimension $M = 36$. To properly accommodate the mixed boundary conditions, we construct the subspace using a half-period 2D Fourier sine basis:
$$ \phi_{k,l}(x,y) = \sin\left(k\frac{\pi}{2}x\right) \sin\left(\ell\frac{\pi}{2}y\right), \quad \text{for } 1 \le k,\ell \le 6. $$
This specific basis is chosen because it satisfies the homogeneous Dirichlet conditions on $\Gamma_D$ ($x=0$ and $y=0$) while naturally possessing non-zero spatial derivatives at the boundaries $\Gamma_N$ ($x=1$ and $y=1$). This structural flexibility allows the swarm to satisfy the non-zero Neumann flux $g(x,y)$ without introducing artificial spatial distortions at the domain edges. Furthermore, the exact analytical solution $u_{\text{exact}}$ is explicitly represented by the $k=2, \ell=2$ mode, guaranteeing that the Galerkin spatial truncation error is strictly zero.

From a computational perspective, the spatially varying coefficient $(1+x+y)$ prevents perfect spectral diagonalization. To maintain computational efficiency, the stiffness matrix $K$ and load vector $F$ are assembled offline using numerical quadrature on a dense $300 \times 300$ spatial grid. Consequently, the iterative CBO algorithm executes entirely within the $36$-dimensional coefficient space, evaluating the discrete energy of the whole swarm via vectorized, grid-free algebraic operations.

The numerical results for this mixed boundary value problem are presented in Figure \ref{fig:example2_results}. Because the spatially varying coefficient $(1+x+y)$ prevents perfect spectral diagonalization, this test thoroughly evaluates the algorithm's robustness. 

As shown in the left panel, the consensus energy gap $|J(\mathfrak{m}_n) - J(u^*)|$ exhibits a rapid, stable exponential decay during the initial 400 iterations before smoothly plateauing at a minimal residual error. This confirms that the trace-class stochastic exploration efficiently navigates the complex energy landscape without scattering. Furthermore, the middle and right panels demonstrate a flawless visual agreement between the exact analytical solution and the final CBO consensus state. The algorithm successfully reconstructs the global minimizer, strictly satisfying both the homogeneous Dirichlet conditions on $\Gamma_D$ and the non-zero Neumann flux on $\Gamma_N$, entirely within the discrete active subspace.

\begin{figure}[htbp]
    \centering
    \makebox[\textwidth][c]
    {%
        \includegraphics[width=1.5\textwidth]{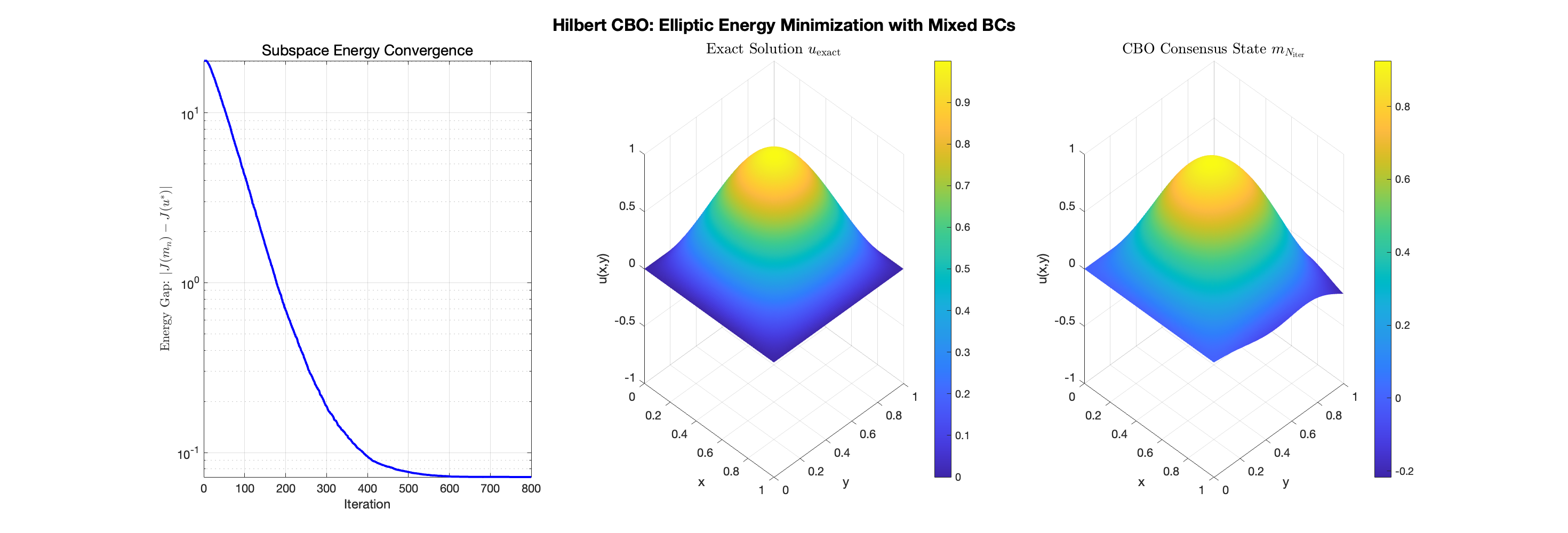}%
    }
    \caption{Numerical evaluation of Hilbert-space CBO for the elliptic energy minimization problem with mixed boundary conditions (Example 2). \textbf{Left:} Logarithmic convergence of the consensus energy gap relative to the exact subspace minimum, showing rapid initial decay followed by stabilization. \textbf{Middle:} Exact solution $u_{\mathrm{exact}}(x,y)=\sin(\pi x)\sin(\pi y)$. \textbf{Right:} Final consensus state $\mathfrak{m}_{N_{\mathrm{iter}}}$ obtained by the algorithm. The close agreement between the numerical and exact solutions demonstrates the ability of the method to handle spatially varying coefficients and nontrivial Neumann boundary data without introducing artificial distortions near the boundary.}
    \label{fig:example2_results}
\end{figure}

\section{Conclusion}

In this work, we established a convergence theory for a time-discrete, finite-particle Consensus-Based Optimization algorithm in a separable Hilbert space. The analysis was carried out directly at the level of the computable, spatially truncated particle system, thereby complementing the continuous-time and mean-field Hilbert-space theory developed in \cite{huang2026derivative}. Building on ideas from the finite-dimensional discrete CBO analysis of \cite{ha2020convergence,ha2021convergence}, we showed that the consensus mechanism extends to the Hilbert-space setting while retaining control with respect to the dimension of the active subspace.

The convergence analysis proceeded in two steps. First, exploiting the common-noise structure of the particle dynamics, we derived quantitative pairwise contraction estimates and established exponential decay of the expected swarm variance. This allowed us to prove that all particles converge almost surely to a common random consensus state $X_\infty\in V$. Second, we addressed the optimization quality of this limiting state. By controlling the evolution of the expected exponentiated energy and employing a quantitative Laplace principle, we showed that, under suitable assumptions on the objective functional and sufficiently well-prepared initial data, the energy of $X_\infty$ can be made arbitrarily close to the global minimum over the active subspace by choosing the inverse temperature parameter sufficiently large.

An important feature of the Hilbert-space formulation is the use of trace-class stochastic exploration. The truncated stochastic variance satisfies $C_{Q,M}\leq \operatorname{Trace}(Q)$ uniformly with respect to the Galerkin dimension $M$, providing a natural mechanism for controlling the exploration as the spatial approximation is refined. The numerical experiments, consisting of an elliptic energy minimization problem with mixed boundary conditions and a PDE-constrained inverse source problem, support the theoretical findings. In particular, they illustrate consensus formation and convergence of the objective values, while the comparison with isotropic finite-dimensional CBO highlights the robustness of trace-class exploration under increasing Galerkin dimension.

Several directions remain open for future investigation. A natural next step is to quantify jointly the optimization error and the spatial approximation error as the Galerkin dimension tends to infinity, thereby connecting the discrete finite-dimensional implementation more directly with the underlying infinite-dimensional optimization problem. It would also be of interest to investigate convergence under more general noise structures and particle-dependent stochastic forcing, as well as adaptive strategies for selecting the active subspace and the covariance spectrum. Finally, extending the present discrete-time analysis to broader classes of constrained, nonsmooth, or genuinely infinite-dimensional optimization problems may further expand the applicability of CBO to variational, inverse, and PDE-constrained optimization.

\section*{Acknowledgments}

This work is partially supported by the Start-up Grant from Hunan University and the Deutsche Forschungsgemeinschaft (DFG, German Research Foundation) with the financial support through HE5386/34-1 Partikelmethoden für unendlich dimensionale Optimierung (561130572) and HE5386/33-1 Control of Interacting Particle Systems, and Their Mean-Field, and Fluid-Dynamic Limits (560288187). Funding from the European Union’s Horizon Europe research and innovation programme under the Marie Sklodowska-Curie Doctoral Network Datahyking (Grant No. 101072546) is acknowledged.

\appendix

\section{Regularity from geometry}\label{app:C11}

Recall that a Fr\'echet differentiable functional $\Phi:H\to\mathbb R$ is said to be $M$-semiconcave if
\[
x\longmapsto \Phi(x)-\frac{M}{2}\|x\|_H^2
\]
is concave, and $M$-semiconvex if
\[
x\longmapsto \Phi(x)+\frac{M}{2}\|x\|_H^2
\]
is convex. Equivalently, by the first-order characterization of concavity and convexity, these become 
\[
\Phi(v) \le \Phi(u) + \langle\nabla\Phi(u),v-u\rangle_H + \frac{M}{2}\|v-u\|_H^2,
\]
and
\[
\Phi(v) \ge \Phi(u) + \langle\nabla\Phi(u),v-u\rangle_H - \frac{M}{2}\|v-u\|_H^2,
\]
respectively.

The estimate \eqref{eq:quadratic_condition} states that the first-order Taylor approximation of $\mathscr{E}$ at $x$ provides a global quadratic upper bound on the energy. Interchanging the roles of $x$ and $y$ in \eqref{eq:quadratic_condition} gives
\[
\mathscr{E}(x)-\mathscr{E}(y) \le \langle\nabla\mathscr{E}(y),x-y\rangle_H +\frac{L_{\mathscr{E}}}{2}\|x-y\|_H^2.
\]
Adding the two inequalities yields
\[
\langle \nabla\mathscr{E}(x)-\nabla\mathscr{E}(y),
x-y \rangle_H \le L_{\mathscr{E}}\|x-y\|_H^2,
\]
which is a one-sided upper Lipschitz estimate for the gradient. Conversely, this inequality is precisely the monotonicity inequality associated with the concavity of $x\mapsto\mathscr{E}(x)-\frac{L_{\mathscr{E}}}{2}\|x\|_H^2$. Hence, \eqref{eq:quadratic_condition} is equivalent to $L_{\mathscr{E}}$-semiconcavity of $\mathscr{E}$.

Similarly, the assumption \eqref{eq:osl_condition} is a one-sided Lipschitz (or hypomonotonicity) condition on the gradient. Indeed, rewriting it as
\[
\langle (\nabla\mathscr{E}(x)+C_{\mathscr{E}}x) - (\nabla\mathscr{E}(y)+C_{\mathscr{E}}y), x-y \rangle_H \ge 0
\]
shows that the mapping $x\mapsto\nabla\mathscr{E}(x)+C_{\mathscr{E}}x$ is monotone. By the first-order characterization of convexity, this is equivalent to the convexity of
\[
x\longmapsto \mathscr{E}(x)+\frac{C_{\mathscr{E}}}{2}\|x\|_H^2,
\]
that is, to the $C_{\mathscr{E}}$-semiconvexity of $\mathscr{E}$.


Therefore, \eqref{eq:osl_condition} when together with \eqref{eq:quadratic_condition} is equivalent to assuming that $\mathscr{E}$ is both semi-convex and semi-concave. 

The following lemma shows that in this case, $\mathscr{E}$ is $C^{1,1}$. Such a result is known in the finite dimensional case (see \cite[Corollary 3.3.8, p.61]{cannarsa2004semiconcave}). We provide below a self-consistent proof in infinite-dimensional Hilbert spaces.

\begin{lemma}
Let $\mathscr{E}:H\to\mathbb R$ be Fr\'echet differentiable. If $\mathscr{E}$ is simultaneously $C$-semiconvex and $L$-semiconcave, then $\mathscr{E}\in C^{1,1}(H)$,  and $\operatorname{Lip}(\nabla\mathscr{E}) \le L+C$. In particular, the conclusion remains true if $\mathscr{E}$ satisfies \eqref{eq:osl_condition}-\eqref{eq:quadratic_condition}.
\end{lemma}

\begin{proof}
Recall a function is said to be $L$-smooth, if its gradient is $L$-Lipschitz. 

\textbf{Step 1.} \textit{(Convexity and semiconcavity imply $L$-smoothness.)}\\
We first show that if a Fr\'echet differentiable function
$\Phi:H\to\mathbb R$ is convex and $M$-semiconcave, then $\|\nabla\Phi(x)-\nabla\Phi(y)\|_H \le M\|x-y\|_H$ for all $x,y\in H.$

If $M=0$, then $\Phi$ is both convex and concave, hence affine, and there is nothing to prove. Assume therefore that $M>0$. Since subtracting a linear functional preserves both convexity and $M$-semiconcavity, the function
\[
\Psi(z):=\Phi(z)-\langle\nabla\Phi(y),z\rangle_H
\]
is convex and $M$-semiconcave. Moreover, $\nabla\Psi(y)=0$,  so  $y$ is a global minimizer of $\Psi$. By $M$-semiconcavity,
\[
\Psi(v) \le \Psi(u) + \langle\nabla\Psi(u),v-u\rangle_H + \frac{M}{2}\|v-u\|_H^2, \qquad u,v\in H.
\]
Indeed, this is the first-order characterization of the concavity of
$\Psi-\frac{M}{2}\|\cdot\|_H^2$. Applying this inequality with
\[
u=x, \qquad v=x-\frac1M\nabla\Psi(x),
\]
yields
\[
\begin{aligned}
    \Psi(v)
    & \le \Psi(x) - \frac{1}{M}\|\nabla\Psi(x)\|_H^2 + \frac{M}{2} \left\| \frac1M\nabla\Psi(x) \right\|_H^2 \\
    & = \Psi(x) - \frac{1}{2M} \|\nabla\Psi(x)\|_H^2.
\end{aligned}
\]
Since $y$ minimizes $\Psi$, $\Psi(y)\le\Psi(v)$, and therefore
\[
\Psi(y) \le \Psi(x) - \frac1{2M} \|\nabla\Psi(x)\|_H^2.
\]
Recalling the definition of $\Psi$, we have $\nabla \Psi(x) = \nabla \Phi(x) - \nabla \Phi(y)$, and using the previous inequality we obtain
\[
\Phi(x)-\Phi(y) - \langle\nabla\Phi(y),x-y\rangle_H \ge \frac{1}{2M} \|\nabla\Phi(x)-\nabla\Phi(y)\|_H^2.
\]
Interchanging the roles of $x$ and $y$ and adding the two inequalities
gives
\[
\langle \nabla\Phi(x)-\nabla\Phi(y), x-y \rangle_H \ge \frac{1}{M} \|\nabla\Phi(x)-\nabla\Phi(y)\|_H^2.
\]
Finally, the Cauchy--Schwarz inequality implies
\[
\|\nabla\Phi(x)-\nabla\Phi(y)\|_H^2 \le M \|\nabla\Phi(x)-\nabla\Phi(y)\|_H \|x-y\|_H,
\]
and hence
\[
\|\nabla\Phi(x)-\nabla\Phi(y)\|_H \le M\|x-y\|_H.
\]

\textbf{Step 2.} \textit{(Semiconvex and semiconcave imply $C^{1,1}$.)}\\
Define
\[
\mathscr{F}(x) := \mathscr{E}(x)+\frac{C}{2}\|x\|_H^2, \qquad \mathscr{G}(x) := \frac{L}{2}\|x\|_H^2-\mathscr{E}(x).
\]
By the $C$-semiconvexity and $L$-semiconcavity of $\mathscr{E}$, both $\mathscr{F}$ and $\mathscr{G}$ are convex. Moreover,
\[
\mathscr{F}+\mathscr{G} = \frac{L+C}{2}\|\cdot\|_H^2.
\]
Consequently, $\mathscr{F}$ is $(L+C)$-semiconcave, since
\[
\mathscr{F}(x)-\frac{L+C}{2}\|x\|_H^2 = -\mathscr{G}(x)
\]
is concave. Similarly,
\[
\mathscr{G}(x)-\frac{L+C}{2}\|x\|_H^2 = -\mathscr{F}(x),
\]
so $\mathscr{G}$ is also $(L+C)$-semiconcave.

If $L+C=0$, then necessarily $L=C=0$. In this case, $\mathscr{E}$ is both convex and concave, hence affine, and $\nabla\mathscr{E}$ is constant. We may therefore assume that $L+C>0$.

Applying Step 1 to $\mathscr{F}$ and $\mathscr{G}$, with $M=L+C$, yields for all $x,y\in H$
\[
\|\nabla\mathscr{F}(x)-\nabla\mathscr{F}(y)\|_H \le (L+C)\|x-y\|_H
\quad \text{ and } \quad 
\|\nabla\mathscr{G}(x)-\nabla\mathscr{G}(y)\|_H \le (L+C)\|x-y\|_H.
\]

Let
\[
p:=\nabla\mathscr{E}(x)-\nabla\mathscr{E}(y), \qquad h:=x-y.
\]
Then
\[
\nabla\mathscr{F}(x)-\nabla\mathscr{F}(y) = p+Ch, \qquad \text{ and } \qquad  \nabla\mathscr{G}(x)-\nabla\mathscr{G}(y) = Lh-p.
\]
Since
\[
p = \frac{L}{L+C}(p+Ch) - \frac{C}{L+C}(Lh-p),
\]
the triangle inequality gives
\[
\begin{aligned}
    \|p\|_H
    & \le \frac{L}{L+C}\|p+Ch\|_H + \frac{C}{L+C}\|Lh-p\|_H \\
    & = \frac{L}{L+C}\|\nabla\mathscr{F}(x)-\nabla\mathscr{F}(y)\|_H + \frac{C}{L+C}\|\nabla\mathscr{G}(x)-\nabla\mathscr{G}(y)\|_H \\
    & \le \frac{L}{L+C}(L+C)\|h\|_H + \frac{C}{L+C}(L+C)\|h\|_H 
     = (L+C)\|h\|_H.
\end{aligned}
\]
Hence,
\[
\|\nabla\mathscr{E}(x)-\nabla\mathscr{E}(y)\|_H \le \bigl(L+C_{\mathscr{E}}\bigr)\|x-y\|_H, \qquad x,y\in H.
\]
Therefore, $\nabla\mathscr{E}$ is globally Lipschitz continuous and $\mathscr{E}\in C^{1,1}(H)$. 
\end{proof}

\begin{remark}
The argument in Step 1 of the above proof is closely related to the classical Baillon-Haddad theorem from convex optimization (see \cite[Corollaire 10]{baillon1977quelques}, \cite{bauschke2010baillon}, \cite[Corollary 18.17, p.323]{bauschke2011convex}). The latter states that if a convex function has an $M$-Lipschitz continuous gradient, then its gradient is $\frac{1}{M}$-cocoercive. In the proof above, we proceed in the opposite direction: starting from $M$-semiconcavity (equivalently, the quadratic upper bound condition), we first derive the cocoercivity inequality 
\[
\langle \nabla\Phi(x)-\nabla\Phi(y), x-y \rangle_H \ge \frac{1}{M} \|\nabla\Phi(x)-\nabla\Phi(y)\|_H^2,
\]
and then recover the $M$-Lipschitz continuity of the gradient by the Cauchy-Schwarz inequality. Thus, the present argument may be viewed as a converse counterpart to the classical Baillon-Haddad implication, with semiconcavity replacing the a priori smoothness assumption.
\end{remark}

\bibliography{bibliography}

@article{ha2021convergence,
  title={Convergence and error estimates for time-discrete consensus-based optimization algorithms},
  author={Ha, Seung Yeal and Jin, Shi and Kim, Doheon},
  journal={Numerische Mathematik},
  volume={147},
  number={2},
  pages={255--282},
  year={2021},
  publisher={Springer New York}
}

@book{scalora1958abstract,
  title={Abstract martingale convergence theorems},
  author={Scalora, Frank Salvatore},
  year={1958},
  publisher={University of Illinois at Urbana-Champaign}
}

@article{bauschke2010baillon,
  title={The {B}aillon-{H}addad Theorem Revisited},
  author={Bauschke, Heinz H and Combettes, Patrick L},
  journal={Journal of Convex Analysis},
  volume={17},
  number={3\&4},
  pages={781--787},
  year={2010}
}

@article{baillon1977quelques,
  title={Quelques propri{\'e}t{\'e}s des op{\'e}rateurs angle-born{\'e}s et n-cycliquement monotones},
  author={Baillon, Jean-Bernard and Haddad, Georges},
  journal={Israel Journal of Mathematics},
  volume={26},
  number={2},
  pages={137--150},
  year={1977},
  publisher={Springer}
}

@misc{bauschke2011convex,
  title={Convex Analysis and Monotone Operator Theory in {H}ilbert Spaces},
  author={Bauschke, Heinz H and Combettes, Patrick L},
  year={2011},
  publisher={Springer Publishing Company, Incorporated}
}

@book{cannarsa2004semiconcave,
  title={Semiconcave Functions, {H}amilton-{J}acobi Equations, and Optimal Control},
  author={Cannarsa, Piermarco and Sinestrari, Carlo},
  year={2004},
  publisher={Springer}
}

@article{kang2026time,
  title={Time-Discrete Consensus-Based Optimization Algorithm for Multi-Objective Optimization Problems},
  author={Kang, Myeongju and Bae, Hyeong-Ohk and Ha, Seung-Yeal and Min, Chanho and Yoo, Jane and Yoon, Wook},
  journal={Mathematical Models and Methods in Applied Sciences},
  year={2026},
  publisher={World Scientific}
}

@article{bianchi2025consensus,
  title={Consensus-based optimization beyond finite-time analysis},
  author={Bianchi, Pascal and Dragomir, Radu-Alexandru and Priser, Victor},
  journal={arXiv preprint arXiv:2509.12907},
  year={2025}
}

@article{beddrich2026constrained,
  title={Constrained consensus-based optimization and numerical heuristics for the few particle regime},
  author={Beddrich, Jonas and Chenchene, Enis and Fornasier, Massimo and Huang, Hui and Wohlmuth, Barbara},
  journal={Journal of Global Optimization},
  pages={1--54},
  year={2026},
  publisher={Springer}
}

@article{aceves2026consensus,
  title={Consensus-based optimization with $\alpha$-stable jump processes},
  author={Aceves-Sanchez, Pedro and Albi, Giacomo and Ferrarese, Federica and Herty, Michael},
  journal={arXiv preprint arXiv:2604.05626},
  year={2026}
}

@article{bungert2025polarized,
  title={Polarized consensus-based dynamics for optimization and sampling},
  author={Bungert, Leon and Roith, Tim and Wacker, Philipp},
  journal={Mathematical Programming},
  volume={211},
  number={1},
  pages={125--155},
  year={2025},
  publisher={Springer}
}

@article{borghi2023adaptive,
  title={An adaptive consensus based method for multi-objective optimization with uniform Pareto front approximation},
  author={Borghi, Giacomo and Herty, Michael and Pareschi, Lorenzo},
  journal={Applied Mathematics \& Optimization},
  volume={88},
  number={2},
  pages={58},
  year={2023},
  publisher={Springer}
}

@article{borghi2026variational,
  title={Variational inference via {G}aussian interacting particles in the {B}ures-{W}asserstein geometry},
  author={Borghi, Giacomo and Carrillo, Jos{\'e} A},
  journal={arXiv preprint arXiv:2601.00632},
  year={2026}
}

@article{borghi2025model,
  title={Model predictive control strategies using consensus-based optimization},
  author={Borghi, Giacomo and Herty, Michael},
  journal={Mathematical Control and Related Fields},
  volume={15},
  number={3},
  pages={876--894},
  year={2025},
  publisher={American Institute of Mathematical Sciences}
}

@article{lyu2025consensus,
  title={Consensus {B}ased {S}tochastic {C}ontrol},
  author={Lyu, Liyao and Chen, Jingrun},
  journal={arXiv preprint arXiv:2501.17801},
  year={2025}
}

@article{kalise2023consensus,
  title={Consensus-based optimization via jump-diffusion stochastic differential equations},
  author={Kalise, Dante and Sharma, Akash and Tretyakov, Michael V},
  journal={Mathematical Models and Methods in Applied Sciences},
  volume={33},
  number={02},
  pages={289--339},
  year={2023},
  publisher={World Scientific}
}

@article{de2025mean,
  title={Mean-Field Model for Two-Layer Neural Networks Trained with Consensus-Based Optimization},
  author={De Deyn, William and Herty, Michael and Samaey, Giovanni},
  journal={arXiv preprint arXiv:2511.21466},
  year={2025}
}

@article{riedl2024leveraging,
  title={Leveraging memory effects and gradient information in consensus-based optimisation: On global convergence in mean-field law},
  author={Riedl, Konstantin},
  journal={European Journal of Applied Mathematics},
  volume={35},
  number={4},
  pages={483--514},
  year={2024},
  publisher={Cambridge University Press}
}

@incollection{totzeck2021trends,
  title={Trends in consensus-based optimization},
  author={Totzeck, Claudia},
  booktitle={Active Particles, Volume 3: Advances in Theory, Models, and Applications},
  pages={201--226},
  year={2021},
  publisher={Springer}
}

@article{khatab2026consensus,
  title={A Consensus-based optimization algorithm using {G}aussian processes for global optimization problems in {S}obolev spaces},
  author={Khatab, Mahmoud and Totzeck, Claudia},
  journal={arXiv preprint arXiv:2603.15337},
  year={2026}
}

@article{fornasier2026consensus,
  title={From Consensus-Based Optimization to Evolution Strategies: {P}roof of Global Convergence},
  author={Fornasier, Massimo and Huang, Hui and Klemenc, Jona and Malaspina, Greta},
  journal={arXiv preprint arXiv:2602.11677},
  year={2026}
}

@article{totzeck2020consensus,
  title={Consensus-based global optimization with personal best},
  author={Totzeck, Claudia and Wolfram, Marie-Therese},
  journal={Mathematical biosciences and engineering: MBE},
  volume={17},
  number={5},
  pages={6026--6044},
  year={2020}
}

@article{ha2020convergence,
  title={Convergence of a first-order consensus-based global optimization algorithm},
  author={Ha, Seung-Yeal and Jin, Shi and Kim, Doheon},
  journal={Mathematical Models and Methods in Applied Sciences},
  volume={30},
  number={12},
  pages={2417--2444},
  year={2020},
  publisher={World Scientific}
}

@article{gerber2025uniform,
  title={Uniform-in-time propagation of chaos for Consensus-Based Optimization},
  author={Gerber, Nicolai and Hoffmann, Franca and Kim, Dohyeon and Vaes, Urbain},
  journal={arXiv preprint arXiv:2505.08669},
  year={2025}
}

@article{cipriani2022zero,
  title={Zero-inertia limit: from particle swarm optimization to consensus-based optimization},
  author={Cipriani, Cristina and Huang, Hui and Qiu, Jinniao},
  journal={SIAM J. Math. Anal.},
  volume={54},
  number={3},
  pages={3091--3121},
  year={2022},
  publisher={SIAM}
}

@article{bonandin2025consensus,
  title={Consensus-based algorithms for stochastic optimization problems},
  author={Bonandin, Sabrina and Herty, Michael},
  journal={SIAM Journal on Optimization},
  volume={35},
  number={4},
  pages={2572--2598},
  year={2025},
  publisher={SIAM}
}

@article{garcia2025defending,
  title={Defending against diverse attacks in federated learning through consensus-based bi-level optimization},
  author={Garc{\'\i}a Trillos, Nicol{\'a}s and Kumar Akash, Aditya and Li, Sixu and Riedl, Konstantin and Zhu, Yuhua},
  journal={Philosophical Transactions of the Royal Society A: Mathematical, Physical and Engineering Sciences},
  volume={383},
  number={2298},
  year={2025},
  publisher={The Royal Society}
}

@article{roith2025consensus,
  title={Consensus-based optimization for closed-box adversarial attacks and a connection to evolution strategies},
  author={Roith, Tim and Bungert, Leon and Wacker, Philipp},
  journal={arXiv preprint arXiv:2506.24048},
  year={2025}
}

@article{wei2025consensus,
  title={A consensus-based optimization method for nonsmooth nonconvex programs with approximated gradient descent scheme},
  author={Wei, Jiazhen and Wu, Fan and Bian, Wei},
  journal={Journal of Global Optimization},
  pages={1--32},
  year={2025},
  publisher={Springer}
}

@article{chen2022consensus,
  author = {Chen, Jingrun and Jin, Shi and Lyu, Liyao},
  title = {A Consensus-Based Global Optimization Method with Adaptive Momentum Estimation},
  journal = {Communications in Computational Physics},
  year = {2022},
  volume = {31},
  number = {4},
  pages = {1296-1316},
  doi = {10.4208/cicp.OA-2021-0144},
  url = {https://doi.org/10.4208/cicp.OA-2021-0144}
}

@article{byeon2025consensus,
  title={Consensus, error estimates and applications of first and second-order consensus-based optimization algorithms},
  author={Byeon, Junhyeok and Ha, Seung-Yeal and Hwang, Gyuyoung and Ko, Dongnam and Yoon, Jaeyoung},
  journal={Math. Models Methods Appl. Sci.},
  year={2025},
  publisher={World Scientific}
}

@article{fornasier2024truncated,
  title={Consensus-based optimisation with truncated noise},
  author={Fornasier, Massimo and Richt{\'a}rik, Peter and Riedl, Konstantin and Sun, Lukang},
  journal={European Journal of Applied Mathematics},
  pages={1--24},
  year={2024},
  publisher={Cambridge University Press}
}

@article{fornasier2025pde,
  title={A {PDE} framework of consensus-based optimization for objectives with multiple global minimizers},
  author={Fornasier, Massimo and Sun, Lukang},
  journal={Communications in Partial Differential Equations},
  pages={1--42},
  year={2025},
  publisher={Taylor \& Francis}
}

@article{ko2022convergence,
  title={Convergence analysis of the discrete consensus-based optimization algorithm with random batch interactions and heterogeneous noises},
  author={Ko, Dongnam and Ha, Seung-Yeal and Jin, Shi and Kim, Doheon},
  journal={Math. Models Methods Appl. Sci.},
  volume={32},
  number={06},
  pages={1071--1107},
  year={2022},
  publisher={World Scientific}
}

@article{bungert2025mirrorcbo,
  title={Mirror{CBO}: A consensus-based optimization method in the spirit of mirror descent},
  author={Bungert, Leon and Hoffmann, Franca and Kim, Doh Yeon and Roith, Tim},
  journal={arXiv preprint arXiv:2501.12189},
  year={2025}
}

@article{huang2023global,
  title={On the global convergence of particle swarm optimization methods},
  author={Huang, Hui and Qiu, Jinniao and Riedl, Konstantin},
  journal={Applied Mathematics \& Optimization},
  volume={88},
  number={2},
  pages={30},
  year={2023},
  publisher={Springer}
}

@article{borghi2023constrained,
	title={Constrained consensus-based optimization},
	author={Borghi, Giacomo and Herty, Michael and Pareschi, Lorenzo},
	journal={SIAM Journal on Optimization},
	volume={33},
	number={1},
	pages={211--236},
	year={2023},
	publisher={SIAM}
}

@article{carrillo2023fedcbo,
	title={Fed{CBO}: Reaching group consensus in clustered federated learning through consensus-based optimization},
	author={Carrillo, Jos{\'e} A and Trillos, Nicolas Garcia and Li, Sixu and Zhu, Yuhua},
	journal={Journal of machine learning research},
	volume={25},
	number={214},
	pages={1--51},
	year={2024}
}

@article{chenchene2025consensus,
  title={A consensus-based algorithm for non-convex multiplayer games},
  author={Chenchene, Enis and Huang, Hui and Qiu, Jinniao},
  journal={Journal of Optimization Theory and Applications},
  volume={206},
  number={2},
  pages={45},
  year={2025},
  publisher={Springer}
}

@article{herty2025multiscale,
  title={A multiscale consensus-based algorithm for multilevel optimization},
  author={Herty, Michael and Huang, Yuyang and Kalise, Dante and Kouhkouh, Hicham},
  journal={Mathematical Models and Methods in Applied Sciences},
  pages={1--37},
  year={2025},
  publisher={World Scientific}
}

@article{ha2022stochastic,
	title={Stochastic consensus dynamics for nonconvex optimization on the {S}tiefel manifold: Mean-field limit and convergence},
	author={Ha, Seung-Yeal and Kang, Myeongju and Kim, Dohyun and Kim, Jeongho and Yang, Insoon},
	journal={Math. Models Methods Appl. Sci.},
	volume={32},
	number={03},
	pages={533--617},
	year={2022},
	publisher={World Scientific}
}

@article{carrillo2022consensus,
	title={Consensus-based sampling},
	author={Carrillo, Jos{\'e} A and Hoffmann, Franca and Stuart, Andrew M and Vaes, Urbain},
	journal={Studies in Applied Mathematics},
	volume={148},
	number={3},
	pages={1069--1140},
	year={2022},
	publisher={Wiley Online Library}
}

@inproceedings{borghi2022consensus,
	title={A consensus-based algorithm for multi-objective optimization and its mean-field description},
	author={Borghi, Giacomo and Herty, Michael and Pareschi, Lorenzo},
	booktitle={2022 {IEEE} 61st {C}onference on {D}ecision and {C}ontrol},
	pages={4131--4136},
	year={2022},
	organization={IEEE}
}

@article{bellavia2025discrete,
  title={A discrete consensus-based global optimization method with noisy objective function},
  author={Bellavia, Stefania and Malaspina, Greta},
  journal={Journal of Optimization Theory and Applications},
  volume={206},
  number={1},
  pages={20},
  year={2025},
  publisher={Springer}
}

@article{fornasier2020consensus,
	title={Consensus-based optimization on hypersurfaces: Well-posedness and mean-field limit},
	author={Fornasier, Massimo and Huang, Hui and Pareschi, Lorenzo and S{\"u}nnen, Philippe},
	journal={Math. Models Methods Appl. Sci.},
	volume={30},
	number={14},
	pages={2725--2751},
	year={2020},
	publisher={World Scientific}
}

@article{carrillo2021consensus,
  title={A consensus-based global optimization method for high dimensional machine learning problems},
  author={Carrillo, Jos{\'e} A and Jin, Shi and Li, Lei and Zhu, Yuhua},
  journal={ESAIM Control Optim. Calc. Var.},
  volume={27},
  pages={S5},
  year={2021},
  publisher={EDP Sciences}
}

@article{fornasier2024consensus,
	title={Consensus-based optimization methods converge globally},
	author={Fornasier, Massimo and Klock, Timo and Riedl, Konstantin},
	journal={SIAM Journal on Optimization},
	volume={34},
	number={3},
	pages={2973--3004},
	year={2024},
	publisher={SIAM}
}

@article{fornasier2022anisotropic,
	title={Anisotropic diffusion in consensus-based optimization on the sphere},
	author={Fornasier, Massimo and Huang, Hui and Pareschi, Lorenzo and S{\"u}nnen, Philippe},
	journal={SIAM Journal on Optimization},
	volume={32},
	number={3},
	pages={1984--2012},
	year={2022},
	publisher={SIAM}
}

@article{carrillo2018analytical,
	title={An analytical framework for consensus-based global optimization method},
	author={Carrillo, Jos{\'e} A and Choi, Young-Pil and Totzeck, Claudia and Tse, Oliver},
	journal={Math. Models Methods Appl. Sci.},
	volume={28},
	number={06},
	pages={1037--1066},
	year={2018},
	publisher={World Scientific}
}

@article{pinnau2017consensus,
	title={A consensus-based model for global optimization and its mean-field limit},
	author={Pinnau, Ren{\'e} and Totzeck, Claudia and Tse, Oliver and Martin, Stephan},
	journal={Math. Models Methods Appl. Sci.},
	volume={27},
	number={01},
	pages={183--204},
	year={2017},
	publisher={World Scientific}
}

@article{borghi2024particle,
  title={A particle consensus approach to solving nonconvex-nonconcave min-max problems},
  author={Borghi, Giacomo and Huang, Hui and Qiu, Jinniao},
  journal={SIAM Journal on Control and Optimization},
  volume={64},
  number={3},
  pages={1573--1601},
  year={2026},
  publisher={SIAM}
}

@article{huang2024consensus,
  title={Consensus-based optimization for saddle point problems},
  author={Huang, Hui and Qiu, Jinniao and Riedl, Konstantin},
  journal={SIAM Journal on Control and Optimization},
  volume={62},
  number={2},
  pages={1093--1121},
  year={2024},
  publisher={SIAM}
}

@article{huang2024self,
  title={Self-interacting {CBO}: {E}xistence, uniqueness, and long-time convergence},
  author={Huang, Hui and Kouhkouh, Hicham},
  journal={Applied Mathematics Letters},
  pages={109372},
  year={2024},
  publisher={Elsevier}
}

@article{choi2025modified,
  title={A modified Consensus-Based Optimization model: consensus formation and uniform-in-time propagation of chaos},
  author={Choi, Young-Pil and Lee, Seungchan and Song, Sihyun},
  journal={arXiv preprint arXiv:2511.19116},
  year={2025}
}

@article{huang2025faithful,
  title={Faithful global convergence for the rescaled consensus-based optimization},
  author={Huang, Hui and Kouhkouh, Hicham and Sun, Lukang},
  journal={(To appear in SIAM J. Optim.) arXiv:2503.08578},
  year={2025}
}

@article{huang2025uniform,
  title={Uniform-in-time mean-field limit estimate for the {C}onsensus-{B}ased {O}ptimization},
  author={Huang, Hui and Kouhkouh, Hicham},
  journal={ESAIM: Control, Optimisation and Calculus of Variations},
  volume={31},
  pages={69},
  year={2025},
  publisher={EDP Sciences}
}

@article{huang2024fast,
  title={Fast and robust consensus-based optimization via optimal feedback control},
  author={Huang, Yuyang and Herty, Michael and Kalise, Dante and Kantas, Nikolas},
  journal={arXiv preprint arXiv:2411.03051},
  year={2024}
}

@article{bonandin2026exploiting,
  title={Exploiting Structure with Anisotropic Consensus-Based Optimization},
  author={Bonandin, Sabrina and Riedl, Konstantin and Veneruso, Sara},
  journal={arXiv preprint arXiv:2607.10205},
  year={2026}
}

@article{bonandin2025strong,
  title={Strong global convergence of the consensus-based optimization algorithm},
  author={Bonandin, Sabrina and Riedl, Konstantin and Veneruso, Sara},
  journal={arXiv preprint arXiv:2512.10654},
  year={2025}
}

@article{huang2026derivative,
  title={A derivative-free particle method for optimization in {H}ilbert spaces},
  author={Huang, Hui and Kouhkouh, Hicham},
  journal={arXiv preprint arXiv:2605.31565},
  year={2026}
}

@article{huang2026collective,
  title={Collective Optimization on {R}iemannian Manifolds with Bounded Curvature},
  author={Huang, Hui and Kim, Dohyun and Park, Hansol},
  journal={arXiv preprint arXiv:2606.14884},
  year={2026}
}

@article{bayraktar2026uniform,
  title={Uniform-in-time weak propagation of chaos for consensus-based optimization},
  author={Bayraktar, Erhan and Ekren, Ibrahim and Zhou, Hongyi},
  journal={The Annals of Applied Probability},
  volume={36},
  number={3},
  pages={2387--2426},
  year={2026},
  publisher={Institute of Mathematical Statistics}
}
\end{document}